\documentclass[12pt]{iopart}
\usepackage{amsmath}
\usepackage{amssymb}  

\usepackage[hidelinks]{hyperref} 

\usepackage{mathrsfs}
\usepackage{amssymb}
\usepackage{amsfonts}
\usepackage{graphicx}
\usepackage{tikz}
\usepackage{amsfonts}
\usepackage{amssymb}
\usepackage{amsthm}
\usepackage{graphicx}
\usepackage{tikz}
\usepackage{pgfplots}
\usepackage{float}
\usepackage{enumitem}
\usepackage{amsthm,thmtools}
\usepackage{thm-restate}

\usetikzlibrary{positioning,arrows.meta,calc}

\pgfplotsset{compat=1.18}

\newenvironment{restated}[1]{%
  \par\noindent\textbf{Theorem #1.}\itshape\ }%
  {\par}

\newcommand{\drawarrow}[3]{%
  \node[circle,fill,inner sep=1.2pt] at (#1) {};
  \draw[very thick, >=stealth, ->, #3] (#1) -- ++(#2);
  \node[circle,fill,inner sep=1.2pt] at ($(#1) + (#2)$) {};
}

\usetikzlibrary{matrix,backgrounds}
\pgfdeclarelayer{myback}

\newtheorem{theorem}{Theorem}[section]
\newtheorem{lemma}[theorem]{Lemma}
\newtheorem{proposition}[theorem]{Proposition}
\newtheorem{corollary}[theorem]{Corollary}

\newtheorem{definition}[theorem]{Definition}
\newtheorem{example}[theorem]{Example}

\newtheorem{notation}[theorem]{Notation}

\newcommand*{\dist}{d}
\newcommand*{\bbbchi}{\boldsymbol{\chi}}
\newcommand*{\bkappa}{\boldsymbol{\kappa}}
\newcommand*{\bbbxi}{\boldsymbol{\xi}}

\newcommand*{\diam}{\mathrm{diam}}

\newcommand*{\dimh}{\dim_\mathrm{H}}
\newcommand*{\dimdeg}{\dim_\mathrm{D}}
\newcommand*{\ximetric}{(\Xi_\iota,\hat\dist_{\Xi_\iota})}

\begin{document}

\title[Reducible Iterated Graph Systems]{Reducible Iterated Graph Systems: \\
multiscale-freeness and multifractals}

\author{Nero Ziyu Li$^{1,2,*}$, Frank Xin Hu$^3$, Thomas Britz$^3$}
\address{$^1$ 
Department of Mathematics,
Imperial College London, United Kingdom}
\address{$^2$ CDT Random Systems, University of Oxford, United Kingdom}
\address{$^3$ School of Mathematics and Statistics, UNSW Sydney, Australia}
\address{$^*$ Author to whom any correspondence should be addressed.}

\ead{z5222549@zmail.unsw.edu.au, frankhumath@gmail.com and britz@unsw.edu.au}

\vspace{0.2cm}\hspace{1.7cm}\small{Mathematics subject classification: 28A20, 05C82}

\begin{abstract}
Iterated Graph Systems (IGS) transplant ideas from fractal geometry into graph theory. 
Building on this framework, we extend Edge IGS from the primitive to the reducible setting. 
Within this broader context, we formulate rigorous definitions of 
multifractality and multiscale-freeness for { fractal graphs}, 
and we establish conditions that are equivalent to the occurrence of these two phenomena. 
We further determine the corresponding fractal and degree spectra, 
proving that both are finite and discrete. 
These results complete the foundational theory of Edge IGS 
by filling the gap left by the primitive case studied in~\cite{li2024scale,neroli2024fractal}.
\end{abstract}

\section{Introduction}
\label{sec:Introduction}

We aim to investigate the fundamental properties of graphs that exhibit fractal characteristics, 
and to this end we introduce the Iterated Graph Systems (IGS) model. 
More generally, one may view an IGS as a form of discrete fractal structure,
an object whose antecedents have in fact been present in the literature for a long time.

\subsection{Background}

The terms ``fractal lattice" or ``hierarchical lattice" were introduced by physicists around 1970 to describe self-similar networks 
on which renormalisation techniques could be applied~\cite{migdal1996recursion,kadanoff1976notes}.
These lattices were soon used to study a wide range of statistical physics models, 
including the Ising and Potts models~\cite{stanley1977cluster,berker1979renormalisation,kaufman1981exactly}, 
percolation and cluster formation~\cite{alexander1982density,gefen1984phase}, 
random walks and anomalous diffusion~\cite{toussaint1983particle,stanley1984flow}, 
as well as diffusion-limited reactions and dynamical processes~\cite{toussaint1983particle,derrida1983fractal}.
Although fractal lattices have been extensively explored from a physical perspective, 
much of that progress is empirical or heuristic and thus falls short of a general and rigorous framework.

The notion of fractal networks, used to describe large empirical graphs in complex-network science, emerged in the early 2000s.~\cite{song2005self}.
Alongside the well-known small-world property of short paths with high clustering~\cite{watts1998collective} 
and the scale-free power-law degree distribution~\cite{barabasi1999emergence}, 
many natural networks were later shown, via box-covering analysis, to be fractal and self-similar across scales~\cite{song2005self}.
Further studies of hierarchical modularity and epidemic spreading confirmed that 
small-world, scale-free and fractal structures now form the canonical triad of 
complex-network properties~\cite{ravasz2003hierarchical,pastor2001epidemic}.

The expression ``fractal graph" is less common in mathematics, 
yet one can trace it back to the 1980s 
when Barlow and Bass rigorously constructed Brownian motion on the Sierpiński gasket and carpet, 
establishing sub-Gaussian heat kernel bounds and foundational results in analysis 
on fractal graphs~\cite{barlow1988brownian,barlow1997transition,barlow1999brownian}.
Around the same time, 
Kigami introduced a framework for defining Laplacian via graph approximations 
and developed the theory of 
post-critically finite self-similar sets~\cite{kigami1989harmonic,kigami1993harmonic,kigami2001analysis}.
Ben Hambly investigated the spectral asymptotics of such graphs, 
relating eigenvalue counting functions to local geometry 
and developing the notion of spectral dimension~\cite{hambly2010diffusion,hambly2011asymptotics}.
These ideas were later extended to more general metric–measure spaces and random walks, 
producing a robust framework for heat-kernel estimates and 
functional inequalities~\cite{grigor2008dichotomy,grigor2003heat,li2021self,devyver2023gradient}.

\medskip

While much is known about analysis on specific fractal graphs, a general and flexible framework for constructing and studying such structures remains largely undeveloped.  
Edge IGS attempts to fill that gap by producing generic fractal graphs via a general iterative procedure.
The underlying idea is hardly new, and its traces can be found in many physical models and mathematical constructions.  
A first prototype of the edge-substitution approach already appears in Xi et al. \cite{xi2017fractality,ye2019average}, yet the notion was not formally defined and named until \cite{neroli2024fractal}.  
Since then, a series of papers has continued to develop the framework and its applications \cite{li2018scale,anttila2024constructions,anttila2024iterated,li2025fractal}.

A systematic study of IGS provides a natural vantage point 
from which to examine the geometry and analysis of fractal graphs.  
Moreover, any results obtained for IGS feed directly back into 
statistical physics, complex networks, and the broader theory of fractal graphs.  
For all of these reasons, we believe that developing the theory of IGS is necessary.

\subsection{Reducibility}

{ 

A non-negative matrix $\mathbf X$ is irreducible if for every $i,j$ there exists $n$ such that $[X^{n}]_{ij} > 0$ and is primitive if $\mathbf X^{n}$ is strictly positive for some $n$. 
We call $\mathbf X$ reducible if it is not irreducible. 
Throughout the paper we apply these notions to the mass matrix $\mathbf M$, the distance matrix $\mathbf D$, and the degree matrix $\mathbf N$ in EIGS. 
For instance, we say the EIGS is mass-primitive if and only if $\mathbf M$ is primitive.

In the primitive cases studied in \cite{li2024scale,neroli2024fractal}, all relevant growth rates are governed by a single Perron eigenvalue, and the resulting dimensions have a single scaling exponent. 
By contrast, extending those results from the primitive case to the reducible case remains completely unexplored.
In the reducible case, several primitive blocks may be reachable from the initial colour. 
Blocks with the same maximal spectral radius can also occur along a reachable chain, producing polynomial corrections and multiple scaling regimes, which is technically subtle.

In this paper, we focus on three notions of reducibility: mass-reducibility, distance-reducibility, and degree-reducibility.
These correspond, respectively, to reachability relations governing mass, distance, and degree.
In \cite{li2024scale,neroli2024fractal}, all three were assumed to be primitive, which leads to concise formulae for the relevant fractal dimensions.
However, many classical fractals are themselves reducible in mass, distance, or degree.
The next two examples show that such
an extension is necessary and meaningful.
}

\begin{figure}[ht]
\centering

\begin{tikzpicture}
\begin{scope}[shift={(0, 0)}, scale=0.6]
    \node[draw=none, fill=none, rectangle] at (0,0.26) {\bf\large{Rule\,:}};
    
    \drawarrow{2,1.5}{2,0}{red}
    
    \draw[very thick,>=stealth,->] (5, 1.5) --++ (1, 0);
    
    \drawarrow{7,1.5}{1,0}{blue}
    \drawarrow{8,1.5}{1,0}{red}
    \drawarrow{8,1.5}{0,1}{red}

    \node[draw=none, fill=none, rectangle] at (7,1) {$\beta^+_1$};
    \node[draw=none, fill=none, rectangle] at (9,1) {$\beta^-_1$};
    
    \node[draw=none, fill=none, rectangle] at (11,1.5) {$R_1$};

    \drawarrow{2,-1}{2,0}{blue}
    \draw[very thick,>=stealth,->] (5,-1) --++ (1,0);
    \drawarrow{7,-1}{2,0}{blue}

    \node[draw=none, fill=none, rectangle] at (7,-1.5) {$\beta^+_1$};
    \node[draw=none, fill=none, rectangle] at (9,-1.5) {$\beta^-_1$};
    \node[draw=none, fill=none, rectangle] at (11,-1) {$R_2$};
    
    \draw[ultra thick, dashed] (-2,-2.5) to (15.8,-2.5);

\end{scope}
\begin{scope}[shift={(0, -3.5)}, scale=0.4]
    \drawarrow{0, 0}{1,1}{red}
    \drawarrow{0, 0}{1,-1}{red}

    \node[draw=none, fill=none, rectangle] at (0.5,-3) {$\Xi^0$};
    
    \draw[very thick,>=stealth,->] (3,0) --++ (1,0);

    \drawarrow{5, 0}{2, 2}{blue}
    \drawarrow{5, 0}{2, -2}{blue}
    \drawarrow{7, 2}{2, 1}{red}
    \drawarrow{7, 2}{2, -1}{red}
    \drawarrow{7, -2}{2, 1}{red}
    \drawarrow{7, -2}{2, -1}{red}

    \node[draw=none, fill=none, rectangle] at (7,-3.5) {$\Xi^1$};

    \draw[very thick,>=stealth,->] (10,0) --++ (1,0);

    \drawarrow{12, 0}{2, 2}{blue}
    \drawarrow{12, 0}{2, -2}{blue}
    \drawarrow{14, 2}{2, 1}{blue}
    \drawarrow{14, 2}{2, -1}{blue}
    \drawarrow{14, -2}{2, 1}{blue}
    \drawarrow{14, -2}{2, -1}{blue}
    \drawarrow{16, 3}{2, 0.5}{red}
    \drawarrow{16, 3}{2, -0.5}{red}
    \drawarrow{16, 1}{2, 0.5}{red}
    \drawarrow{16, 1}{2, -0.5}{red}
    \drawarrow{16, -1}{2, 0.5}{red}
    \drawarrow{16, -1}{2, -0.5}{red}
    \drawarrow{16, -3}{2, 0.5}{red}
    \drawarrow{16, -3}{2, -0.5}{red}

    \node[draw=none, fill=none, rectangle] at (16,-4) {$\Xi^2$};

    \draw[very thick,>=stealth,->] (19,0) --++ (1,0);

    \node[draw=none, fill=none, rectangle] at (22,0) {\LARGE{...}};

\end{scope}
\end{tikzpicture} 
    \caption{An example of a reducible EIGS: the binary tree}
    \label{fig:binary_tree}
\end{figure}
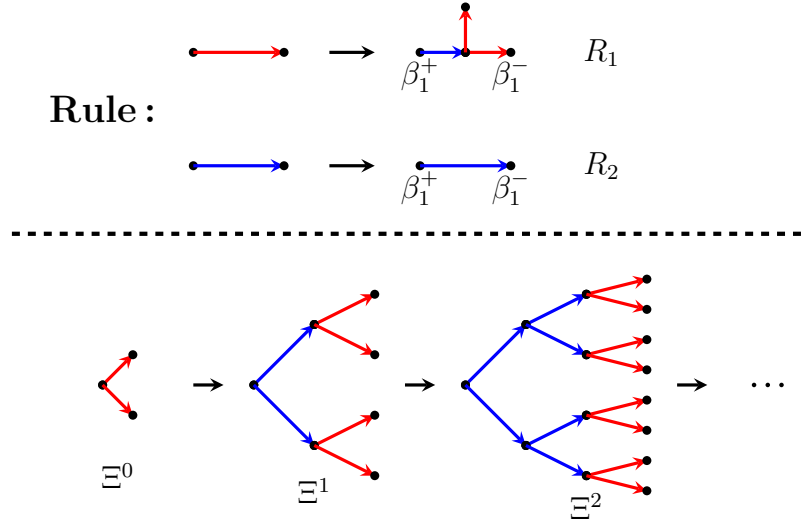

\begin{example}
Edge Iterated Graph Systems work in the following way. 
Starting from an initial graph, one proceeds step by step: 
at each step, every edge is replaced, independently and according to its colour, 
by the rule graph assigned to that colour. 

Figure~\ref{fig:binary_tree} displays a motivating example of a classical fractal { graph}, the binary tree, generated by a { mass-reducible, distance-reducible and degree-reducible EIGS}. 
The initial graph $\Xi^0$ consists of two red edges. 
At every subsequent step each red edge is replaced independently by the rule graph $R_1$, 
while each blue edge is replaced independently by the rule graph $R_2$. 
The binary tree is a well-known example whose Minkowski dimension is infinite; 
precise definitions and a full discussion will be given later.
\end{example}

\begin{figure}[ht]
\centering
    \newcommand{\reddiamond}[3]{%
  \begin{scope}[
    shift={(#1)},
    scale=#2,
    rotate=#3
  ]
    \drawarrow{0,0}{2,0.5}{red}
    \drawarrow{0,0}{2,-0.5}{red}
    \drawarrow{2,0.5}{2,-0.5}{blue}
    \drawarrow{2,-0.5}{2,0.5}{green}
  \end{scope}
}

\newcommand{\bluediamond}[3]{%
  \begin{scope}[
    shift={(#1)},
    scale=#2,
    rotate=#3
  ]
    \drawarrow{0,0}{2,0.5}{blue}
    \drawarrow{0,0}{2,-0.5}{blue}
    \drawarrow{2,0.5}{2,-0.5}{blue}
    \drawarrow{2,-0.5}{2,0.5}{blue}
  \end{scope}
}

\newcommand{\greendiamond}[3]{%
  \begin{scope}[
    shift={(#1)},
    scale=#2,
    rotate=#3
  ]
    \drawarrow{0,0}{2,0.5}{green}
    \drawarrow{0,0}{2,-0.5}{green}
    \drawarrow{2,0.5}{2,-0.5}{green}
    \drawarrow{2,-0.5}{2,0.5}{green}
    \drawarrow{2,-0.5}{0,1}{green}
  \end{scope}
}

\resizebox{\textwidth}{!}{%

\begin{tikzpicture}
\begin{scope}[shift={(0, 0)}, scale=0.6]
    \node[draw=none, fill=none, rectangle] at (0,-.5) {\bf\large{Rule\,:}};
    
    \drawarrow{2,1}{2,0}{red}
    
    \draw[very thick,>=stealth,->] (5, 1) --++ (1, 0);
    
    \drawarrow{7,1}{2,1}{red}
    \drawarrow{7,1}{2,-1}{red}
    \drawarrow{9,2}{2,-1}{blue}
    \drawarrow{9,0}{2,1}{green}

    \node[draw=none, fill=none, rectangle] at (11,0.5) {$\beta_1^-$};
    \node[draw=none, fill=none, rectangle] at (7,0.5) {$\beta_1^+$};
    \node[draw=none, fill=none, rectangle] at (12,1) {$R_1$};

    \drawarrow{2,-2}{2,0}{blue}
    
    \draw[very thick,>=stealth,->] (5, -2) --++ (1, 0);

    \drawarrow{7,-2}{2,1}{blue}
    \drawarrow{7,-2}{2,-1}{blue}
    \drawarrow{9,-1}{2,-1}{blue}
    \drawarrow{9,-3}{2,1}{blue}

    \node[draw=none, fill=none, rectangle] at (11,-2.5) {$\beta_2^-$};
    \node[draw=none, fill=none, rectangle] at (7,-2.5) {$\beta_2^+$};
    \node[draw=none, fill=none, rectangle] at (12,-2) {$R_2$};

    \drawarrow{14,-1}{2,0}{green}
    
    \draw[very thick,>=stealth,->] (17, -1) --++ (1, 0);

    \drawarrow{19,-1}{2,1}{green}
    \drawarrow{19,-1}{2,-1}{green}
    \drawarrow{21,0}{2,-1}{green}
    \drawarrow{21,-2}{2,1}{green}
    \drawarrow{21,-2}{0,2}{green}

    \node[draw=none, fill=none, rectangle] at (23,-1.5) {$\beta_3^-$};
    \node[draw=none, fill=none, rectangle] at (19,-1.5) {$\beta_3^+$};
    \node[draw=none, fill=none, rectangle] at (24,-1) {$R_3$};

    \draw[ultra thick, dashed] (-2,-3.5) to (13,-3.5);

\end{scope}
\begin{scope}[shift={(-2, -3)}, scale=0.5]
    \drawarrow{2,0}{2,0}{red}

    \node[draw=none, fill=none, rectangle] at (3,-2) {$\Xi_1^0$};
    
    \draw[very thick,>=stealth,->] (5, 0) --++ (1, 0);
    
    \reddiamond{7,0}{1}{0}

    \node[draw=none, fill=none, rectangle] at (9,-2) {$\Xi_1^1$};

    \draw[very thick,>=stealth,->] (12, 0) --++ (1, 0);

    \reddiamond{14,0}{0.5}{30}
    \reddiamond{14,0}{0.5}{-30}
    \bluediamond{15.75,1}{0.5}{-30}
    \greendiamond{15.75,-1}{0.5}{30}

    \node[draw=none, fill=none, rectangle] at (15.8,-2) {$\Xi_1^2$};

    \draw[very thick,>=stealth,->] (18, 0) --++ (1, 0);

    \reddiamond{20,0}{0.5}{70}
    \reddiamond{20,0}{0.5}{20}
    \bluediamond{20.7,1.85}{0.5}{20}
    \greendiamond{21.9,0.65}{0.5}{68}
    
    \reddiamond{20,0}{0.5}{-70}
    \reddiamond{20,0}{0.5}{-20}
    \greendiamond{20.7,-1.85}{0.5}{-20}
    \bluediamond{21.9,-0.65}{0.5}{-70}
    
    \bluediamond{22.6,2.55}{0.5}{-20}
    \bluediamond{22.6,2.55}{0.5}{-70}
    \bluediamond{24.5, 1.85}{0.5}{-70}
    \bluediamond{23.3, 0.65}{0.5}{-17}
    
    \greendiamond{22.6,-2.55}{0.5}{20}
    \greendiamond{22.6,-2.55}{0.5}{70}
    \greendiamond{24.5,-1.85}{0.5}{70}
    \greendiamond{23.3,-0.65}{0.5}{20}

    \greendiamond{24.5,-1.85}{0.42}{135}

    \node[draw=none, fill=none, rectangle] at (26,-2) {$\Xi_1^3$};
    
    \draw[very thick,>=stealth,->] (26, 0) --++ (1, 0);

    \node[draw=none, fill=none, rectangle] at (28,0) {\LARGE{...}};
\end{scope}
\end{tikzpicture} 
}
    \caption{A reducible EIGS: the \emph{Splendor} diamond hierarchical lattice}
    \label{fig:splendor}
\end{figure}
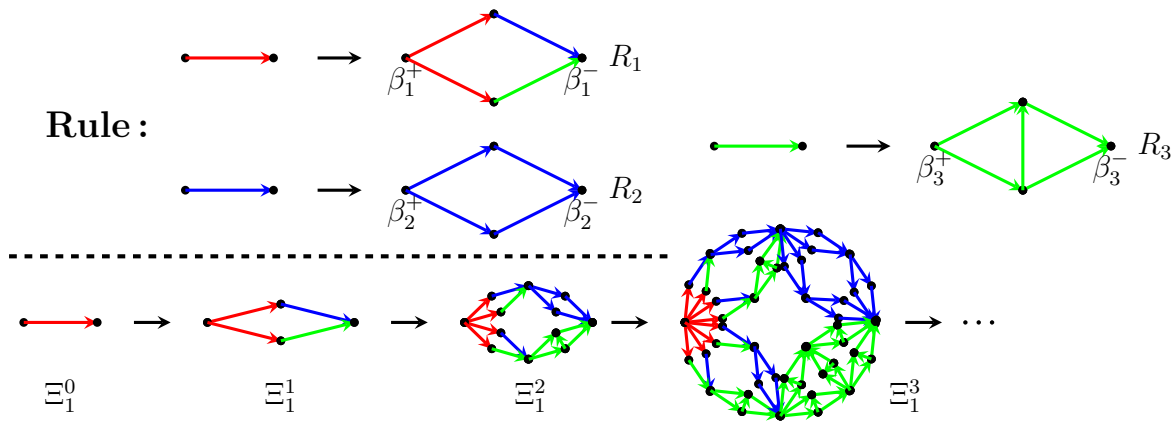
\begin{example}

Figure~\ref{fig:splendor} illustrates the construction of a variation of 
the so-called diamond hierarchical lattice (DHL).
{ Note that any single-coloured IGS is automatically mass-primitive and distance-primitive because the matrices are $1 \times 1$.
However, degree-primitivity is not automatic even when $K=1$, since the degree matrix $\mathbf N$ has size $2\times 2$ and encodes the in/out degrees of the two planting vertices. 
Hence whether the classical DHL is degree-primitive depends on the direction of edges. }

If we have only $R_1$ with only one colour, 
then it turns out to be the classical DHL, 
which is a much-studied example 
whose limit graph has integer Minkowski dimension $2$ and is scale-free with degree dimension $2$. 
See, for instance, 
\cite{hambly2010diffusion,levitan2022renormalization,cruz2023percolation} for detailed analyses.

Our picture, however, uses the reducible rule set $R_1,R_2,R_3$ 
and therefore produces what one may call a {\em splendor diamond hierarchical lattice}. 
The splendor DHL is an instructive example: 
it satisfies both multifractality and multiscale-freeness, 
and, remarkably, its fractal spectrum coincides with its degree spectrum.
\end{example}

One can readily imagine that estimating the growth behaviour of degrees is already a delicate task; 
assessing distances is even more complicated, particularly for an arbitrary reducible IGS.
Many natural questions therefore arise: 
\begin{itemize}
    \item How large is the limiting diameter?
    \item What is the fractal dimension?
    \item Does the scale-free property survive in any case? 
\end{itemize}
Once these questions have been answered, one can hope to investigate further models on such graphs.

The main theorems of this paper, reproduced in the subsection below, answer these questions 
in the wider setting of Reducible Iterated Graph Systems.

\subsection{Main results}

For the first two theorems, we work under the standing assumptions from Section~3.
That is, for every colour $j\in\mathfrak r_{\mathbf M}(\iota)$, the row-selection family $\mathcal D$ is locally primitive stable at $j$, and the matrix $\mathbf M_{\mathrm{dist}}(j)$ is primitive-Frobenius.

The first theorem gives the global Hausdorff dimension of the limit metric space.
Its meaning is that, once we restrict to the distance layer that survives under the metric normalisation, the Hausdorff dimension is obtained by dividing the maximal surviving mass growth rate by the distance growth rate.

\begin{restated}{\ref{thm:singlefractal}}
Let $\mathscr I=(\Xi_\iota^0,\mathcal R,\mathcal S)$ be a reducible edge iterated graph system, and let $(\Xi_\iota,\hat\dist_{\Xi_\iota})$ be its Gromov--Hausdorff scaling limit.
Then
\[
\dimh(\Xi_\iota)
=
\frac{
\displaystyle
\log \max_{k\in\mathfrak R_{\mathbf M_{\mathrm{dist}}(\iota)}(\iota)}
\rho\big(\mathbf B_k(\mathbf M_{\mathrm{dist}}(\iota))\big)
}{
\displaystyle
\log \min_{\mathbf D\in\mathcal D}
\max_{\ell\in\mathfrak R_{\mathbf D}(\iota)}
\rho\big(\mathbf B_\ell(\mathbf D)\big)
}.
\]
\end{restated}

The second theorem describes the full coarse fractal spectrum of the limit metric space.
It says that every Hausdorff dimension that appears on a ball comes from one of the colours in the distance layer of the initial colour, and no other value can occur.

\begin{restated}{\ref{thm:multifractal}}
Let $\mathscr I=(\Xi_\iota^0,\mathcal R,\mathcal S)$ be a reducible edge iterated graph system, and let $(\Xi_\iota,\hat\dist_{\Xi_\iota})$ be its Gromov--Hausdorff scaling limit.
Then
\[
\mathscr S(\Xi_\iota,\hat\dist_{\Xi_\iota})
=
\left\{
\frac{
\displaystyle
\log \max_{k\in\mathfrak R_{\mathbf M_{\mathrm{dist}}(j)}(j)}
\rho\big(\mathbf B_k(\mathbf M_{\mathrm{dist}}(j))\big)
}{
\displaystyle
\log \min_{\mathbf D\in\mathcal D}
\max_{\ell\in\mathfrak R_{\mathbf D}(j)}
\rho\big(\mathbf B_\ell(\mathbf D)\big)
}
:
j\in I_{\mathrm{dist}}(\iota)
\right\}.
\]
Consequently, $(\Xi_\iota,\hat\dist_{\Xi_\iota})$ satisfies the balanced-distance and divergent-mass condition if and only if it is a multifractal with a finite discrete spectrum, that is,
\[
1<
\big|
\mathscr S(\Xi_\iota,\hat\dist_{\Xi_\iota})
\big|
<\infty.
\]
\end{restated}

For the degree side, the key point is different.
The Hausdorff dimension in the metric setting above is a global quantity and therefore only sees the largest surviving growth rate, whereas degree dimension is defined from the exact degree distribution and must therefore distinguish different degree classes.

The third theorem gives the criterion for ordinary scale-freeness.
It says that the degree dimension exists precisely when all surviving degree classes have the same asymptotic mass growth rate, and that common growth rate then determines the degree dimension.

\begin{restated}{\ref{thm:singlescalefree}}
Let $\mathscr I=(\Xi_\iota^0,\mathcal R,\mathcal S)$ be a reducible edge iterated graph system, and let $(\Xi_\iota,\hat{\deg}_{\Xi_\iota})$ be its degree scaling limit.
If
\[
\max_{\mathbf u\in\mathcal U_\iota}
\max_{a:[\mathbf u]_a>0}
\max_{\ell\in\mathfrak R_{\mathbf N}(a)}
\rho\big(\mathbf B_\ell(\mathbf N)\big)
>1,
\]
then $\Xi_\iota$ is scale-free if and only if $\Xi_\iota$ does not satisfy the balanced-degree and divergent-mass condition and
\[
\max_{K\in\mathscr K_{\deg}(\iota)}
\max_{\mathbf u\in K}
\Lambda_{\mathbf M}^{\deg}(\mathbf u)
>1.
\]
In that case,
\[
\dimdeg(\Xi_\iota,\hat{\deg}_{\Xi_\iota})
=
\frac{
\displaystyle
\log \max_{K\in\mathscr K_{\deg}(\iota)}
\max_{\mathbf u\in K}
\Lambda_{\mathbf M}^{\deg}(\mathbf u)
}{
\displaystyle
\log \max_{\mathbf u\in\mathcal U_\iota}
\max_{a:[\mathbf u]_a>0}
\max_{\ell\in\mathfrak R_{\mathbf N}(a)}
\rho\big(\mathbf B_\ell(\mathbf N)\big)
}.
\]
\end{restated}

The fourth theorem records the full degree spectrum.
It says that, even when ordinary scale-freeness fails, every asymptotic branch of the degree distribution is still governed by one of the surviving degree classes, and the whole spectrum is exactly the collection of the corresponding exponents.

\begin{restated}{\ref{thm:multiscalefree}}
Let $\mathscr I=(\Xi_\iota^{0},\mathcal R,\mathcal S)$ be a reducible edge iterated graph system, and let $(\Xi_\iota,\hat{\deg}_{\Xi_\iota})$ be its degree scaling limit.
If
\[
\max_{\mathbf u\in\mathcal U_\iota}
\max_{a:[\mathbf u]_a>0}
\max_{\ell\in\mathfrak R_{\mathbf N}(a)}
\rho\big(\mathbf B_\ell(\mathbf N)\big)
>1,
\]
then
\[
\mathscr D(\Xi_\iota,\hat{\deg}_{\Xi_\iota})
=
\left\{
\frac{
\displaystyle
\log \max_{\mathbf u\in K}
\Lambda_{\mathbf M}^{\deg}(\mathbf u)
}{
\displaystyle
\log \max_{\mathbf u\in\mathcal U_\iota}
\max_{a:[\mathbf u]_a>0}
\max_{\ell\in\mathfrak R_{\mathbf N}(a)}
\rho\big(\mathbf B_\ell(\mathbf N)\big)
}
:
K\in\mathscr K_{\deg}(\iota)
\right\}.
\]
Consequently, $(\Xi_\iota,\hat{\deg}_{\Xi_\iota})$ satisfies the balanced-degree and divergent-mass condition if and only if it is finite discrete multiscale-free.
\end{restated}

\bigskip

Although a reducible IGS is a natural extension of an irreducible one, 
the step from irreducible to reducible turns out to be indispensable and far from trivial.

In forthcoming work, we shall introduce a broader framework called General Iterated Graph Systems, 
which in particular generates many classical objects, such as the Sierpiński gasket graph, 
in an entirely transparent way. 
Most of those classical fractals rely crucially on {\em reducible} substitutions; without a firm understanding of reducible IGS behaviour, we cannot hope to build a general theory of fractal graphs, such as fractal lattices or fractal networks. 
Hence, although parts of this paper appear technical, the analysis carried out here is unavoidable.

\section{Reducible Iterated Graph Systems}
In this section, we use the Lustig--Uyanik theorem to study the role of reducibility in EIGS and to understand to what extent the conclusions from the primitive case remain valid.
Although the reducible setting is more technically involved, the growth still takes the form of a polynomial correction multiplied by an exponential term.
This allows us to analyse the behaviour of the fractal spectrum directly from the underlying graph structure.

\subsection{Construction of reducible IGS}

In this subsection, we introduce the definitions of reducible IGS. 
Since some of these definitions can be found in \cite{li2024scale,neroli2024fractal}, 
we aim to keep them as concise as possible without sacrificing readability.

\begin{notation}
For each positive integer $N\in\mathbb{N}_+$, 
define $[N] := \{1,2,\ldots,N\}$.
In this paper, the variables $i$, $j$, $k$, $m$, and $n$ represent integers, 
with $n$ specifically denoting the number of iterations.

Let $\{\bbbxi_i\}_{i=1}^m $ be the standard basis for $\mathbb{R}^{1\times m}$, 
where each basis vector $\bbbxi_i$ has~$1$ in its $i$-th entry and $0$ in all other entries.
For any vector $\mathbf{v}\in\mathbb{N}^{1\times m}$, 
denote each $i$-th entry of $\mathbf{v}$ as $[\mathbf{v}]_i := \mathbf{v} \bbbxi_i^\top$.

For any matrices $\mathbf{A},\mathbf{B}\in\mathbb{R}^{m\times m}$, 
let $[\mathbf{A}]_{ij}:=\bbbxi_i \mathbf{A} \bbbxi_j^\top$ be the entry in row $i$ and column~$j$ of $\mathbf{A}$.
Let $\mathbf{v} \leq \mathbf{w}$ denote that $[\mathbf{v}]_i \leq [\mathbf{w}]_i$  for all $i \in [m]$ and
let $\mathbf{A} \leq \mathbf{B}$ denote that $[\mathbf{A}]_{ij}\leq [\mathbf{B}]_{ij}$ for all $i,j\in[m]$.
\end{notation}

\begin{definition}
An {\em Edge Iterated Graph System} (EIGS) $\mathscr{I}=(\Xi_{\iota}^0,\mathcal{R},\mathcal{S})$ is a triple where, 
for some positive integer of colours $K$,
\begin{itemize}
    \item[$\Xi_{\iota}^0$] is a given $\iota$-coloured edge where $\iota \in [K]$. 
          When the subscript $\iota$ is omitted, $\Xi^{0}$ may denote an arbitrary directed finite graph.
          In particular, we denote the two vertices in $\Xi_\iota^0$ by $\mathfrak{v}^+$ and $\mathfrak{v}^-$.
          In this paper, $\iota$ always represents the initial colour;
    \item[$\mathcal{R}$] $=\{R_i\}_{i \in [K]}$ is a family of directed connected graphs, 
          called {\em rule} or {\em seed graphs}. 
          Each rule graph must contain two {\em planting vertices} $\{\beta_i^+, \beta_i^-\}\subset V(R_i)$.
\end{itemize}
The colour function $\mathscr{C}(e)$ indicates the colour of edge $e$.
The substitution operator is defined as 
\begin{itemize}
    \item[$\displaystyle\mathcal{S}$] $: e=(a,b) \mapsto R_{\mathscr{C}(e)}$ 
          fixing $a = \beta_{\mathscr{C}(e)}^+,\; b = \beta_{\mathscr{C}(e)}^-$.
\end{itemize}
That is, we substitute the arc $(a,b)$ by $R_{\mathscr{C}(e)}$ such that 
$a$ and $b$ coincide with $\beta_{\mathscr{C}(e)}^+$ and $\beta_{\mathscr{C}(e)}^-$.
In this paper, we require $\dist_{R_i}(\beta_i^+,\beta_i^-) \geq 2$.
Whenever a distance is considered in a directed graph, it is understood to be the graph distance in its underlying undirected graph.

We now define a discrete-time dynamical system indexed by $n\in\mathbb{N}$
\[
\Xi_{\iota}^{n+1}:=\mathcal{S}(\Xi_{\iota}^n):=\bigcup_{e \in E(\Xi_{\iota}^n)} \mathcal{S}(e)  ,
\]
satisfying $\mathcal{S}^0(\Xi_{\iota}^0) = \Xi_{\iota}^0$ 
and $\mathcal{S}^{n+m}(\Xi_{\iota}^0) = \mathcal{S}^n\big(\mathcal{S}^m(\Xi_{\iota}^0)\big)$, 
for all $n,m\in\mathbb{N}$.
Let $\Xi_{\iota}$ be the Gromov-Hausdorff scaling limit defined in~\cite{neroli2024fractal}
of $(\Xi_{\iota}^n)_{n\in\mathbb{N}}$ (see Theorem~\ref{thm:GH-limit} below). 

We here record the notions of reducibility that are already available at this stage, 
where the definitions of $\mathbf{M}$ and $\mathbf{N}$ will be given in the subsequent subsection.
An iterated graph system is said to be {\em mass-reducible} if $\mathbf{M}$ is reducible; 
and it is said to be {\em degree-reducible} if $\mathbf{N}$ is reducible.
The notion of distance-reducibility will be introduced later, after the distinguished matrix $\mathbf D_\iota^*$ has been defined.
\end{definition}

\subsection{Relevant definitions}
In this subsection, we will provide the definitions of $\mathbf{M}$, $\mathbf{N}$, and $\mathcal{D}$.

\begin{notation}
For a directed graph $G$, we define the vector $\bbbchi(G)\in\mathbb{N}^{1 \times K}$ by
\[
  [\bbbchi (G)]_i := \big| \{ e \in E(G) : \mathscr{C}(e)=i \} \big| ,
\]
where $\bbbchi(G)$ counts the number of edges of each colour in the entire graph $G$. 
Specifically, the $i$-th component $[\bbbchi(G)]_i$ represents 
the total number of edges in $G$ that have colour~$i$.

For each vertex $v \in V(G)$, 
the vector $\bkappa_G(v)\in\mathbb{N}^{1 \times 2K}$ is defined as a vector capturing the local edge distribution around $v$:
  \begin{align*}
      [\bkappa_G(v)]_{2i-1} 
  &:= \big|\{ e\text{ is an outgoing edge of $v$} : \mathscr{C}(e)=i \} \big| \quad& \text{for } i \in [K] ;\\
      [\bkappa_G(v)]_{2i}
  &:= \big|\{ e\text{ is an incoming edge of $v$} : \mathscr{C}(e)=i \} \big| \quad& \text{for } i \in [K] .
  \end{align*}

Here, $\bkappa_G(v)$ encodes both the out-degree and in-degree of vertex $v$ for each colour~$i$. 
When the ambient graph is clear, we simply write $\bkappa(v)$.
\end{notation}

\begin{definition}
The matrices $\mathbf{N}\in\mathbb{N}^{2K \times 2K}$ and $\mathbf{M}\in\mathbb{N}^{K \times K}$ are constructed as follows:
\[
      \mathbf{N} := 
      \begin{pmatrix}
    \bkappa_{R_1}(\beta_1^+) \\
    \bkappa_{R_1}(\beta_1^-) \\[-1mm]
    \vdots \\
    \bkappa_{R_K}(\beta_K^+) \\
    \bkappa_{R_K}(\beta_K^-)    
    \end{pmatrix} \qquad
    \mathbf{M} := 
    \begin{pmatrix}
    \bbbchi(R_1) \\[-1mm]
    \vdots \\
    \bbbchi(R_K)   
    \end{pmatrix}  .
\]
\end{definition}
The matrix $\mathbf{N}$ systematically organises the local connectivity information of all planting vertices across all rule graphs in the edge iterated graph system. 


The matrix $\mathbf{M}$ captures the total number of edges of each colour in each rule graph~$R_i$. 
Specifically, the entry $[\mathbf{M}]_{ij}$ represents the number of edges in $R_i$ that have colour~$j$. 
This matrix provides a concise summary of the colour distribution across all seed graphs in the iterated graph system.
See Figure~\ref{fig:frank} for an example.
\begin{figure}[htb]
    \begin{minipage}{0.54\textwidth}
\centering
\begin{tikzpicture}[scale=0.45]
\draw (-5.5,-6) node {\normalsize \textbf{Rule\,:}};
\draw (6,  0) node {\text{$R_1$}};
\draw (6, -3) node {\text{$R_2$}};
\draw (6, -5) node {\text{$R_3$}};
\draw (6, -7) node {\text{$R_4$}};
\draw (6,-10) node {\text{$R_5$}};
\draw (6,-13) node {\text{$R_6$}};

\draw (0,  -0.5) node {\text{$\beta_1^+$}};
\draw (3.5,  -0.5) node {\text{$\beta_1^-$}};

\draw (0,  -3.5) node {\text{$\beta_2^+$}};
\draw (3.5,  -3.5) node {\text{$\beta_2^-$}};

\draw (0,  -5.5) node {\text{$\beta_3^+$}};
\draw (4,  -5.5) node {\text{$\beta_3^-$}};

\draw (0,  -7.5) node {\text{$\beta_4^+$}};
\draw (3.5,  -7.5) node {\text{$\beta_4^-$}};

\draw (0,  -10.5) node {\text{$\beta_5^+$}};
\draw (2.5,  -10.5) node {\text{$\beta_5^-$}};

\draw (0,  -13.5) node {\text{$\beta_6^+$}};
\draw (2.5,  -13.5) node {\text{$\beta_6^-$}};

\tikzstyle{every node}=[circle, draw, fill=black!20,
                        inner sep=0pt, minimum width=4pt]

\tikzstyle{cat}=[fill=none, draw=none]

\foreach \i in {0,-3,-5,-7,-10,-13}{
             \draw[->,>=stealth] (-1.35,\i)--(-0.65,\i);
}

   \foreach \i in {0}{

            \draw[black, thick] (-3,\i)--(-2,\i);

            \node at (-3,\i) {};
            \node at (-2,\i) {};
   }

\draw[black, thick] (0,0)--(1,0);
\draw[black, thick] (1,0)--(2,0);
\draw[black, thick] (1,0)--(1,1);
\draw[black, thick] (1,0)--(1,-1);
\draw[blue,  thick] (2,0)--(3,-0);

\node at (3,0) {};
\node at (0,0) {};
\node at (2,0) {};
\node at (1,0) {};
\node at (1,1) {};
\node at (1,-1) {};


   \foreach \i in {-3}{

            \draw[blue, thick] (-3,\i)--(-2,\i);

            \node at (-3,\i) {};
            \node at (-2,\i) {};
   }

\draw[blue,  thick] (0,-3)--(1,-3);
\draw[green, thick] (1,-3)--(1,-2);
\draw[green, thick] (1,-3)--(2,-3);
\draw[cyan,  thick] (3,-3)--(2,-3);

\node at (0,-3){};
\node at (1,-3){};
\node at (2,-3){};
\node at (1,-2){};

\node at (3,-3){};

   \foreach \i in {-5}{

            \draw[green, thick] (-3,\i)--(-2,\i);

            \node at (-3,\i) {};
            \node at (-2,\i) {};
   }

\draw[blue, thick] (0,-5)--(1,-5);
\draw[blue, thick] (1,-5)--(2,-5);
\draw[blue, thick] (1,-5)--(1,-5);
\draw[blue, thick] (1,-5)--(1,-4);
\draw[yellow, thick] (3,-5)--(3,-4);
\draw[yellow, thick] (3,-5)--(4,-5);
\draw[green, thick] (2,-5)--(3,-5);
\draw[green, thick] (2,-5)--(2,-4);

\foreach \i in {0,1,2,3} {
  \node at (\i,-5) {};
}

\foreach \i in {1,2} {
  \node at (\i,-4) {};
}
\node at (4,-5){};

\node at (3,-4){};

   \foreach \i in {-7}{

            \draw[red, thick] (-3,\i)--(-2,\i);

            \node at (-3,\i) {};
            \node at (-2,\i) {};
   }

\draw[red, thick] (0,-7)--(1,-7);
\draw[yellow, thick] (1,-7)--(2,-7);
\draw[red, thick] (1,-7)--(1,-6);

\draw[cyan, thick] (2,-7)--(3,-7);
\draw[cyan, thick] (2,-7)--(2,-6);
\draw[cyan, thick] (2,-7)--(2,-8);

\foreach \i in {0,1,2,3} {
  \node at (\i,-7) {};
}

\foreach \i in {1,2} {
  \node at (\i,-6) {};
}

\node at (2,-8) {};

   \foreach \i in {-10}{

            \draw[yellow, thick] (-3,\i)--(-2,\i);

            \node at (-3,\i) {};
            \node at (-2,\i) {};
   }

\draw[red, thick] (0,-10)--(1,-10);
\draw[cyan, thick] (1,-10)--(2,-10);
\draw[yellow, thick] (1,-10)--(1,-9);
\draw[yellow, thick] (1,-10)--(1,-11);

\node at (0,-10) {};
\node at (2,-10) {};
\node at (1,-10) {};
\node at (1,-9) {};
\node at (1,-11) {};


   \foreach \i in {-13}{

            \draw[cyan, thick] (-3,\i)--(-2,\i);

            \node at (-3,\i) {};
            \node at (-2,\i) {};
   }

\draw[red, thick] (0,-13)--(1,-13);
\draw[cyan, thick] (1,-13)--(2,-13);
\draw[red, thick] (1,-13)--(1,-12);

\node at (0,-13){};
\node at (1,-13){};
\node at (2,-13){};
\node at (1,-12){};

        %
\end{tikzpicture}
\end{minipage}
\begin{minipage}{0.45\textwidth}

\[ \mathbf{M}=
    \begin{tikzpicture}[baseline=-\the\dimexpr\fontdimen22\textfont2\relax ]
        \matrix [matrix of math nodes,left delimiter=(,right delimiter=)] (m)
{
4 & 1 & 0 & 0 & 0 & 0 \\
0 & 1 & 2 & 0 & 0 & 1 \\
0 & 3 & 2 & 0 & 2 & 0 \\
0 & 0 & 0 & 2 & 1 & 3 \\
0 & 0 & 0 & 1 & 2 & 1 \\
0 & 0 & 0 & 2 & 0 & 1 \\
};
        \draw[thin, draw=white, fill=green, opacity=0.3] (m-1-1.south west) rectangle (m-1-1.north east);
        \draw[thin, draw=white, fill=cyan, opacity=0.3] (m-3-2.south west) rectangle (m-2-3.north east);
        \draw[thin, draw=white, fill=pink, opacity=0.3] (m-6-4.south west) rectangle (m-4-6.north east);
    \end{tikzpicture}
\]

\end{minipage}
    \caption{A reducible iterated graph system}
\label{fig:frank}
\end{figure}
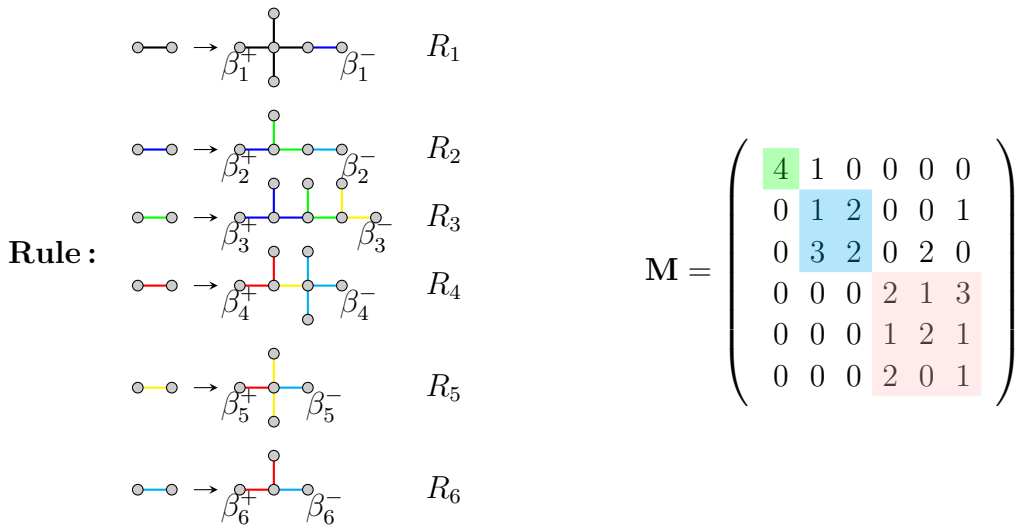  

We now begin to analyse the growth rate of path lengths in an EIGS.
Recall that, whenever we discuss metrics, we transform directed graphs into their underlying undirected graphs.
Throughout this paper, distances are always defined with respect to these underlying undirected graphs.
That is, when considering distances, we will always ignore the direction of edges.

\begin{definition}\label{def:Path Matrix D}
In this paper, paths have no repeating vertices.
Denote one path between vertices $\beta_i^+$ and $\beta_i^-$ in underlying (undirected) $R_i$ by $P_i(\beta_i^+,\beta_i^-)$, and define
\[
  \mathcal{P}_i :=\big\{P_i(\beta_i^+,\beta_i^-) : P_i(\beta_i^+,\beta_i^-) \subset R_i \big\} .
\]
Consider the Cartesian product 
\[
  \mathcal{P} := \prod_{j=1}^{ K} \mathcal{P}_j .
\]
Each element, or {\em choice}, 
$\mathcal{C} = (P_1,\ldots,P_K)\in\mathcal{P}$ is a vector of length $K$ 
each of whose entries $P_i(\beta_i^+,\beta_i^-)$ is in $\mathcal{P}_i$.

For each $i \in [K]$, define
\[
  \mathcal{V}_i := \big\{ \bbbchi\big( P_i(\beta_i^+,\beta_i^-) \big) : P_i(\beta_i^+,\beta_i^-)\in\mathcal{P}_i\big\} .
\]
Let $\mathcal{D}$ be the set of all matrices 
\[
    \mathbf{D}_{\mathcal{C}}
  := \begin{pmatrix}\mathbf{d}_1\\
  \vdots\\
  \mathbf{d}_K\end{pmatrix}
  = \begin{pmatrix}\bbbchi([\mathcal{C}]_1)\\
  \vdots\\
  \bbbchi([\mathcal{C}]_K)\end{pmatrix} 
\]
where $\mathbf{d}_i\in \mathcal{V}_i$ for each $i \in [K]$ or, 
equivalently, $[\mathcal{C}]_i\in\mathcal{P}_i$. 
\end{definition}

\subsection{Reducible forms}

\begin{definition}\label{def:frobenius-form}
A non-negative matrix $\mathbf X\in\mathbb R^{K\times K}$ is
\emph{reducible} if there exists a permutation matrix $\mathbf P$
and an integer $\mathfrak h_{\mathbf X}\ge 2$ such that
\[
  \mathbf P\mathbf X\mathbf P^\top 
  = \begin{pmatrix}
      \mathbf B_1(\mathbf X) & \mathbf X_{12}         & \cdots & \mathbf X_{1\mathfrak h_{\mathbf X}}\\
      0                      & \mathbf B_2(\mathbf X) & \cdots & \mathbf X_{2\mathfrak h_{\mathbf X}}\\
      \vdots                 & \vdots                 & \ddots & \vdots\\
      0                      & 0                      & \cdots & \mathbf B_{\mathfrak h_{\mathbf X}}(\mathbf X)
     \end{pmatrix},
\]
where each diagonal block $\mathbf B_k(\mathbf X)$,
$k\in[\mathfrak h_{\mathbf X}]$, is irreducible.
The displayed form is called the {\em Frobenius normal form}.
If every block $\mathbf B_k(\mathbf X)$ is primitive,
then we call $\mathbf X$ \emph{primitive-Frobenius}.
Throughout this paper, we assume that all matrices are {\em primitive-Frobenius.}
This implies that every block is itself primitive.

\smallskip
We partition the index set $[K]$ into
$\mathfrak h_{\mathbf X}$ blocks
$\mathfrak B_1,\ldots,\mathfrak B_{\mathfrak h_{\mathbf X}}$,
each corresponding to a diagonal block $\mathbf B_k(\mathbf X)$.
Let
\[
   \mathfrak b_{\mathbf{X}}:[K]\longrightarrow[\mathfrak h_{\mathbf X}],
   \qquad
   \mathfrak{b}_{\mathbf{X}} (i)=k
   \text{ iff } i\in\mathfrak B_k .
\]
Write $i\rightsquigarrow_{\mathbf X} j$
if there exists $n\in\mathbb{N}$ such that $(\mathbf X^n)_{ij}>0$.
For a colour $\iota\in[K]$, set
\[
   \mathfrak r_{\mathbf X}(\iota)
   :=\{ i\in[K]:\iota\rightsquigarrow_{\mathbf X} i \} .
\]
For blocks $\mathfrak B_k,\mathfrak B_\ell$ we write
$\mathfrak B_k\rightsquigarrow_{\mathfrak B}\mathfrak B_\ell$
if some $i\in\mathfrak B_k$ and $j\in\mathfrak B_\ell$
satisfy $i\rightsquigarrow_{\mathbf X} j$.
Finally, for any colour $i\in[K]$ define the set of \emph{reachable blocks}
\[
   \mathfrak R_{\mathbf X}(i)
   :=\big\{
        k\in[\mathfrak h_{\mathbf X}]
        :\mathfrak b_{\mathbf X}(i)\rightsquigarrow_{\mathfrak B} k
     \big\} .
\]

When $\mathbf X=\mathbf N\in\mathbb R^{2K\times 2K}$, all of the above notions are understood with the ambient index set $[2K]$ in place of $[K]$.
In particular, $\mathfrak b_{\mathbf N}:[2K]\to[\mathfrak h_{\mathbf N}]$, and for each $\iota_{\deg}\in[2K]$ the reachable-index set and the reachable-block set are defined exactly as above with $\iota_{\deg}\in[2K]$ in place of $i\in[K]$.
When comparing $\mathbf N$ with $\mathbf M$, we also use the projection
\[
\pi:[2K]\to[K],
\qquad
\pi(2i-1)=\pi(2i)=i.
\]
\end{definition}

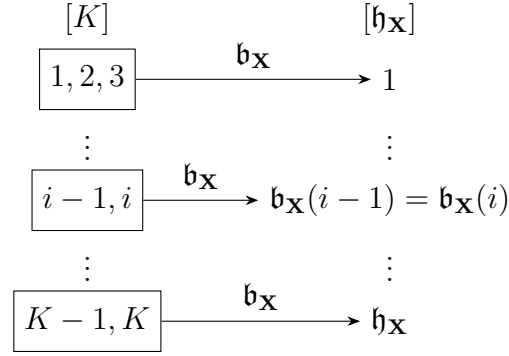
\begin{figure}

\begin{center} 
\begin{tikzpicture}[
  node distance=0.8cm and 2cm, 
  on grid,
  align=center,
  >={Stealth} 
]

\node (headerK) [] {$[K]$};
\node (headerBlank) [right=of headerK] {};
\node (headerHX) [right=of headerBlank] {$[\mathfrak{h}_{\mathbf{X}}]$};

\node (K) [draw, minimum size=0.8cm, below=of headerK] {$1, 2, 3$};
\node (dots1) [below=of K] {$\vdots$};
\node (mid) [draw, minimum size=0.8cm, below=of dots1] {$i-1, i$};
\node (dots2) [below=of mid] {$\vdots$};
\node (Kend) [draw, minimum size=0.8cm, below=of dots2] {$K-1, K$};

\node (blank) [right=of K] {};

\node (hX) [right=of blank] {1};
\node (dots3) [below=of hX] {$\vdots$};
\node (hX_mid) [below=of dots3] {$\mathfrak{b}_{\mathbf X}(i-1)=\mathfrak{b}_{\mathbf X}(i)$};
\node (dots4) [below=of hX_mid] {$\vdots$};
\node (hX_end) [below=of dots4] {$\mathfrak{h}_{\mathbf{X}}$};

\draw[->] (K.east) -- node[above] {$\mathfrak{b}_{\mathbf{X}}$} (hX.west);
\draw[->] (mid.east) -- node[above] {$\mathfrak{b}_{\mathbf{X}}$} (hX_mid.west);
\draw[->] (Kend.east) -- node[above] {$\mathfrak{b}_{\mathbf{X}}$} (hX_end.west);
\end{tikzpicture}
\end{center}
    \caption{An illustration of block maps}
    \label{fig:block-maps}
\end{figure}

\begin{example}
An example of a reducible iterated graph system is shown in Figure~\ref{fig:frank}.
The transition matrix $\mathbf{M}$ corresponding to this system is written into block form, 
with blocks $\mathfrak{B}_1=\{1\}, \mathfrak{B}_2=\{2,3\}, \mathfrak{B}_3=\{4,5,6\}$.
\end{example}

In this paper, we utilize a significant extension of the Perron-Frobenius Theorem 
that is applicable to reducible non-negative matrices. 
We summarise the theorem, which appears relatively straightforward but has a non-trivial proof. 
We will present this result using the notation defined previously.

\begin{theorem}[Lustig-Uyanik Theorem \cite{lustig2016perron}]\label{thm:LU}
Let $\mathbf X\in\mathbb R^{K\times K}$ be 
a primitive-Frobenius matrix written in upper-block-triangular Frobenius normal form, 
and fix a colour $\iota\in[K]$ with block index $\mathfrak b_{\mathbf X}(\iota)\in[\mathfrak h_{\mathbf X}]$.
Let
\begin{align*}
\varkappa_{\mathbf X}(\iota) :=
\max\Big\{
   \Big|
     \big\{
       0\le j\le L:
       \rho\big(\mathbf B_{k_j}(\mathbf X)\big)
       =
       \max_{\ell\in\mathfrak R_{\mathbf X}(\iota)}
       \rho\big(\mathbf B_\ell(\mathbf X)\big)
     \big\}
   \Big| \\
   :
   \mathfrak b_{\mathbf X}(\iota)=k_0<k_1<\cdots<k_L,
   \ 
   \mathfrak B_{k_0}\rightsquigarrow_{\mathfrak B}\cdots
   \rightsquigarrow_{\mathfrak B}\mathfrak B_{k_L}
\Big\}.
\end{align*}
There exist positive constants $c_1,c_2$ depending only on $\mathbf X$
such that, for every integer $n\ge 1$,
\begin{equation}\label{eq:LU-bounds}
   c_1 
   n^{\varkappa_{\mathbf X}(\iota)-1} 
   \Big(
    \max_{k\in\mathfrak R_{\mathbf X}(\iota)}
    \rho\big(\mathbf B_k(\mathbf X)\big)
    \Big)^n
     \le  
   \Big\|\bbbxi_\iota \mathbf X^n\Big\|_1
     \le  
   c_2 
   n^{\varkappa_{\mathbf X}(\iota)-1} 
   \Big(
    \max_{k\in\mathfrak R_{\mathbf X}(\iota)}
    \rho\big(\mathbf B_k(\mathbf X)\big)
    \Big)^n.
\end{equation}
Moreover,
\begin{equation}\label{eq:LU-direction}
   \lim_{n\to\infty}
   \frac{\bbbxi_\iota \mathbf X^n}
        {\big\|\bbbxi_\iota \mathbf X^n\big\|_1}
    = 
   \mathbf v_\iota\neq\mathbf 0 .
\end{equation}
\smallskip
In particular, when $\mathbf X=\mathbf N\in\mathbb R^{2K\times 2K}$, the same statement holds with the ambient index set $[2K]$ in place of $[K]$.
That is, for each ${\iota_{\deg}}\in[2K]$, if $\varkappa_{\mathbf N}(\iota_{\deg})$ is defined by the same formula as above with $\iota_{\deg}$ in place of $\iota$, then there exist positive constants $c_1,c_2$ depending only on $\mathbf N$ such that, for every integer $n\ge 1$,
\[
   c_1 
   n^{\varkappa_{\mathbf N}(\iota_{\deg})-1} 
   \Big(
    \max_{k\in\mathfrak R_{\mathbf N}(\iota_{\deg})}
    \rho\big(\mathbf B_k(\mathbf N)\big)
    \Big)^n
     \le  
   \Big\|\bbbxi_{\iota_{\deg}} \mathbf N^n\Big\|_1
     \le  
   c_2 
   n^{\varkappa_{\mathbf N}(\iota_{\deg})-1} 
   \Big(
    \max_{k\in\mathfrak R_{\mathbf N}(\iota_{\deg})}
    \rho\big(\mathbf B_k(\mathbf N)\big)
    \Big)^n.
\]
Moreover,
\[
   \lim_{n\to\infty}
   \frac{\bbbxi_{\iota_{\deg}} \mathbf N^n}
        {\big\|\bbbxi_{\iota_{\deg}} \mathbf N^n\big\|_1}
    = 
   \mathbf v_{\iota_{\deg}}\neq\mathbf 0 .
\]
\end{theorem}

The current literature including \cite{lustig2016perron} does not 
explicitly establish the existence of the norm limit for reducible non-negative matrices. 
In particular, equations \eqref{eq:LU-bounds} and \eqref{eq:LU-direction} only provide 
two-sided estimates for the length of the iterates and guarantee convergence of their direction; 
they do not identify the actual limit of the length itself. 
Since our argument relies on that limit, 
we present the following lemma, which states its existence in precise form.

We write $f \overset{x\to x_0}{\asymp} g$ if $\lim_{x\to x_0} f(x)/g(x)=c$ where $c \in \mathbb R_+$.
\begin{lemma}\label{lemma:asymp-value}
$\displaystyle\big\Vert \bbbxi_\iota \mathbf{X}^n\big\Vert_1
\asymp 
    n^{\varkappa_{\mathbf X}(\iota)-1}
    \Big(
    \max_{k\in\mathfrak R_{\mathbf X}(\iota)}
    \rho\big(\mathbf B_k(\mathbf X)\big)
    \Big)^n
$.
\end{lemma}
\begin{proof}
Replacing $\mathbf X$ by the principal submatrix indexed by $\mathfrak r_{\mathbf X}(\iota)$, and relabelling the indices if necessary, we may assume that every diagonal block of $\mathbf X$ is reachable from $\mathfrak b_{\mathbf X}(\iota)$.
This does not change $\|\bbbxi_\iota \mathbf X^n\|_1$, $\varkappa_{\mathbf X}(\iota)$, or $\max_{k\in \mathfrak R_{\mathbf X}(\iota)} \rho(\mathbf B_k(\mathbf X))$.

Set
$
\lambda := \max_{k\in \mathfrak R_{\mathbf X}(\iota)} \rho(\mathbf B_k(\mathbf X)).
$
Let $\mathbf{1}$ be the all-ones column vector of the appropriate size.
Since $\mathbf X$ is block upper triangular, its spectrum is the union of the spectra of the diagonal blocks $\mathbf B_k(\mathbf X)$.
For every diagonal block $\mathbf B_k(\mathbf X)$ with $\rho(\mathbf B_k(\mathbf X)) = \lambda$, primitivity implies that $\lambda$ is the only eigenvalue of $\mathbf B_k(\mathbf X)$ on the circle $\{ z \in \mathbb{C} : |z| = \lambda \}$.
Hence $\lambda$ is the only eigenvalue of $\mathbf X$ with modulus $\lambda$.

Let $\mathbf X = \mathbf S \mathbf J \mathbf S^{-1}$ be the Jordan decomposition of $\mathbf X$ over $\mathbb{C}$.
For a Jordan block $\mathbf J_s(\mu) = \mu \mathbf I_s + \mathbf U_s$, where $\mathbf U_s$ is nilpotent, one has
\[
\mathbf J_s(\mu)^n = \sum_{q=0}^{s-1} \binom{n}{q} \mu^{n-q} \mathbf U_s^q.
\]
Therefore every entry of $\mathbf X^n$, and hence the scalar $\bbbxi_\iota \mathbf X^n \mathbf{1}$, is a finite linear combination of terms of the form $n^q \mu^n$.
The terms with $\mu = \lambda$ combine into a polynomial in $n$ multiplied by $\lambda^n$, while every term with $|\mu| < \lambda$ is $o(n^d \lambda^n)$ for every fixed $d$.
If the polynomial corresponding to $\mu = \lambda$ were identically zero, then $\|\bbbxi_\iota \mathbf X^n\|_1 = o(\lambda^n)$, which contradicts the lower bound in Theorem \ref{thm:LU}.
Hence there exist an integer $d \ge 0$ and a constant $c \neq 0$ such that
\[
\|\bbbxi_\iota \mathbf X^n\|_1 = \bbbxi_\iota \mathbf X^n \mathbf{1} = c n^d \lambda^n + o(n^d \lambda^n).
\]

Theorem \ref{thm:LU} yields positive constants $c_1$ and $c_2$ such that
\[
c_1 n^{\varkappa_{\mathbf X}(\iota)-1} \lambda^n
\le
\|\bbbxi_\iota \mathbf X^n\|_1
\le
c_2 n^{\varkappa_{\mathbf X}(\iota)-1} \lambda^n
\]
for all $n \ge 1$.
Comparing this with the asymptotic formula above shows that
$
d = \varkappa_{\mathbf X}(\iota) - 1.
$

The lower bound above implies that $\|\bbbxi_\iota \mathbf X^n\|_1 > 0$ for all sufficiently large $n$.
Hence the leading coefficient $c$ must satisfy $c > 0$.
Therefore
\[
\frac{\|\bbbxi_\iota \mathbf X^n\|_1}{n^{\varkappa_{\mathbf X}(\iota)-1} \lambda^n} \to c \in \mathbb{R}_+.
\]
By our notation, this is exactly
$
\|\bbbxi_\iota \mathbf X^n\|_1 \asymp n^{\varkappa_{\mathbf X}(\iota)-1} \lambda^n.
$
\end{proof}

\begin{corollary}\label{cor:row-vector-asymp}
Let $\mathbf u\in\mathbb R_{\ge 0}^{1\times 2K}\setminus\{\mathbf 0\}$.
Then
\[
\big\|\mathbf u \mathbf N^n\big\|_1
\asymp
\max_{a:[\mathbf u]_a>0}
\left\{
n^{\varkappa_{\mathbf N}(a)-1}
\Big(
\max_{k\in\mathfrak R_{\mathbf N}(a)}
\rho\big(\mathbf B_k(\mathbf N)\big)
\Big)^n
\right\}.
\]
\end{corollary}

\begin{proof}
Write
$
\mathbf u=\sum_{a:[\mathbf u]_a>0}[\mathbf u]_a\bbbxi_a.
$
By Lemma~\ref{lemma:asymp-value},
\[
\big\|\mathbf u \mathbf N^n\big\|_1
=
\sum_{a: [\mathbf u]_a>0} [\mathbf u]_a\big\|\bbbxi_a \mathbf N^n\big\|_1
\asymp
\max_{a:[\mathbf u]_a>0}\big\|\bbbxi_a \mathbf N^n\big\|_1.
\]
The conclusion follows immediately.
\end{proof}

\subsection{Preliminaries}

In~\cite{li2024scale,neroli2024fractal}, the irreducible cases for the following lemmas have been proved.
Here, we briefly state the lemmas required for the reducible case.

The results in the following lemma are not listed directly in~\cite{li2024scale,neroli2024fractal} 
but can be obtained by easy generalization. 
Therefore, we present the lemma directly without proof.

\begin{lemma}\label{lemma:preivous_results}
\begin{enumerate}[label=(\roman*)]
    \item $\bkappa (v_{\iota}^{m,n})= \bkappa(v_{\iota}^{m,m}) \mathbf{N}^{n-m}$, 
          where $v_{\iota}^{m,n}$ is a vertex in $\Xi^n$ that appears for the first time in $\Xi^m$ for $n \geq m$.
    \item $\big|E(\Xi_{\iota}^n)\big|=\big\Vert \bbbxi_\iota  \mathbf{M}^n \big\Vert_1 $.
    \item $\displaystyle\big|V(\Xi_{\iota}^n)\big|\asymp \sum_{i=0}^n \big\Vert \bbbxi_\iota  \mathbf{M}^i \big\Vert_1$.
    \item $\dist_{\Xi_\iota^n}(\mathfrak{v}^+,\mathfrak{v}^-) 
    = \min_{\mathbf D_1,\cdots,\mathbf D_n \in \mathcal D} \big\Vert \bbbxi_\iota\prod_{k=1}^n \mathbf D_k\big\Vert_1$.
    \item $\displaystyle\diam(\Xi_\iota^n) \asymp \sum_{i=1}^n \dist_{\Xi_\iota^i}(\mathfrak{v}^+,\mathfrak{v}^-)$ 
    where $\diam(*)$ denotes the diameter of $*$ .
\end{enumerate}
\end{lemma}

Using the Lustig-Uyanik Theorem, Lemma~\ref{lemma:asymp-value} and Corollary~\ref{cor:row-vector-asymp} and Lemma~\ref{lemma:preivous_results}, 
it is easy to prove the following lemma.

\begin{lemma}\label{lem:basic-properties}$ $
\begin{enumerate}
    \item $\displaystyle
    \deg(v_\iota^{m,n})
\asymp
\max_{\iota_{\deg}:[\bkappa(v_\iota^{m,m})]_{\iota_{\deg}}>0}
\left\{
(n-m)^{\varkappa_{\mathbf N}(\iota_{\deg})-1}
\Big(
\max_{k\in\mathfrak R_{\mathbf N}(\iota_{\deg})}
\rho\big(\mathbf B_k(\mathbf N)\big)
\Big)^{n-m}
\right\}$
    \item 
    $\displaystyle
    \big|E(\Xi_{\iota}^n)\big|
    \asymp  
    n^{\varkappa_{\mathbf{M}}(\iota)-1}\max_{k\in\mathfrak{R}_{\mathbf M}(\iota)} 
    \rho \big( \mathbf{B}_k (\mathbf{M}) \big)^n
    $.
    \item 
    $\displaystyle
    \big|V(\Xi_{\iota}^n)\big|
    \asymp  
    n^{\varkappa_{\mathbf{M}}(\iota)-1}\max_{k\in\mathfrak{R}_{\mathbf M}(\iota)} 
    \rho \big( \mathbf{B}_k (\mathbf{M}) \big)^n
    $.
\end{enumerate}
\end{lemma}

\begin{proposition}\label{prop:distance-to-mass}
For any $\mathbf D \in \mathcal D$, every primitive block of $\mathbf D$ is contained in a unique primitive block of $\mathbf M$.
In particular, $\mathfrak h_{\mathbf D} \ge \mathfrak h_{\mathbf M}$.
\end{proposition}

\begin{proof}
Since each chosen path $P_i(\beta_i^+,\beta_i^-)$ is a subgraph of $R_i$, one has $\mathbf D \le \mathbf M$ entrywise.
Hence $\mathbf D^n \le \mathbf M^n$ for every $n \ge 1$.

Let $C$ be a primitive block of $\mathbf D$.
If $i,j \in C$, then there exist $m,n \ge 1$ such that
\[
[\mathbf D^m]_{ij} > 0
\quad \text{and} \quad
[\mathbf D^n]_{ji} > 0.
\]
Therefore
\[
[\mathbf M^m]_{ij} > 0
\quad \text{and} \quad
[\mathbf M^n]_{ji} > 0.
\]
Thus $i$ and $j$ lie in the same primitive block of $\mathbf M$.

So every primitive block of $\mathbf D$ is contained in a unique primitive block of $\mathbf M$.
Equivalently, the block partition induced by $\mathbf D$ refines the block partition induced by $\mathbf M$.
{ 
Hence $\mathfrak h_{\mathbf D} \ge \mathfrak h_{\mathbf M}$.}
\end{proof}

\begin{proposition}\label{prop:degree-to-mass}
Every primitive block of $\mathbf N$ projects under $\pi$ into a unique primitive block of $\mathbf M$.
In particular, $\mathfrak h_{\mathbf N} \ge \mathfrak h_{\mathbf M}$.
\end{proposition}

\begin{proof}
If $[\mathbf N]_{ab} > 0$, then the corresponding entry counts edges of colour $\pi(b)$ incident to the planting vertex indexed by $a$ in the rule graph $R_{\pi(a)}$.
Such an edge is, in particular, an edge of colour $\pi(b)$ in $R_{\pi(a)}$.
Hence
\[
[\mathbf N]_{ab} > 0
\implies
[\mathbf M]_{\pi(a),\pi(b)} > 0.
\]
Consequently, every directed path in the digraph of $\mathbf N$ projects under $\pi$ to a directed path in the digraph of $\mathbf M$.

Let $C$ be a primitive block of $\mathbf N$.
If $a,b \in C$, then $a$ and $b$ are mutually reachable in $\mathbf N$.
Therefore $\pi(a)$ and $\pi(b)$ are mutually reachable in $\mathbf M$.
So $\pi(C)$ is contained in a unique primitive block of $\mathbf M$.

For each primitive block $B$ of $\mathbf M$, choose any colour $i \in B$ and any index $a \in \{2i-1,2i\}$.
Let $C(a)$ be the primitive block of $\mathbf N$ containing $a$.
Since $i=\pi(a)\in \pi(C(a))$, the previous paragraph shows that $\pi(C(a)) \subseteq B$.

If two distinct primitive blocks of $\mathbf M$ produced the same block $C(a)$ of $\mathbf N$, then $\pi(C(a))$ would meet both of them, which is impossible.
Thus distinct primitive blocks of $\mathbf M$ give distinct primitive blocks of $\mathbf N$.
Hence $\mathfrak h_{\mathbf N} \ge \mathfrak h_{\mathbf M}$.
\end{proof}

\section{Fractal dimensions and multifractal spectrum}

In this section, our goal is to determine the coarse fractal spectrum of $\Xi_\iota$.
To do so, we first study the growth rate of distances.
A tile survives with positive diameter only when its colour stays in the same distance layer.
This leads to a Hausdorff dimension formula and to a finite discrete coarse fractal spectrum.
The core object in this section is therefore the limit metric space rather than the finite approximating graphs.
We analyse which descendant tiles remain visible after normalising the metric by the diameter, and which ones collapse to points.
Once the metric structure is understood, the Hausdorff dimension is recovered from the number and size of the surviving tiles.
This is a global quantity, so only the maximal surviving mass growth rate matters.
The coarse fractal spectrum is then obtained by examining Hausdorff dimensions of balls in the limit metric space.

\subsection{Growth of distances}

We begin with the metric side of the problem.
For each colour, the relevant quantity is the smallest possible exponential growth rate of the distance between the two planting vertices under repeated substitution.
The point of the local primitive stability assumption is that this optimal growth can still be controlled by a single primitive trimmed matrix.
This allows us to combine a Bellman-type minimisation with Perron--Frobenius theory.
Once this is done, both the planted distance and the diameter grow at the same exponential rate.

Fix an initial colour $\iota\in[K]$.
Throughout this section, for each colour $j\in\mathfrak r_{\mathbf M}(\iota)$, let
$R_j:=\bigcup_{\mathbf D\in\mathcal D}\mathfrak r_{\mathbf D}(j)$ and
$A_j(\mathbf D):=[\mathbf D]_{R_j\times R_j}$.
Also write
\[
\Lambda_{\mathcal D}(j)
:=
\min_{\mathbf D\in\mathcal D}
\max_{\ell\in\mathfrak R_{\mathbf D}(j)}
\rho\big(\mathbf B_\ell(\mathbf D)\big).
\]
We assume that, for every $j\in\mathfrak r_{\mathbf M}(\iota)$, there exists
$\mathbf D_j^*\in\mathcal D$ such that
$\rho\big(A_j(\mathbf D_j^*)\big)=\Lambda_{\mathcal D}(j)$, the matrix
$A_j(\mathbf D_j^*)$ is primitive, and replacing one row of $\mathbf D_j^*$ indexed by $R_j$
by the corresponding row of any other matrix in $\mathcal D$ still gives a primitive trimmed matrix.
We write $f(x) \succeq g(x)$ if there exists a constant $c>0$ independent of $x$, such that $f(x) \geq { c} g(x)$ for all $x$ in the domain.

\begin{theorem}\label{theorem:distance}
For every colour $j\in\mathfrak r_{\mathbf M}(\iota)$ there exists a constant $c_j\in\mathbb R_+$ such that
\[
\frac{1}{\Lambda_{\mathcal D}(j)^n}
\min_{\mathbf D_1,\ldots,\mathbf D_n\in\mathcal D}
\bigg\|
\bbbxi_j\prod_{m=1}^n \mathbf D_m
\bigg\|_1
\to c_j
\]
as $n\to\infty$.
Moreover, for every sequence $\mathbf D_1,\ldots,\mathbf D_n$ with $\mathbf D_m\in\mathcal D$ for each $m\in[n]$,
\[
\bigg\|
\bbbxi_j\prod_{m=1}^n \mathbf D_m
\bigg\|_1
\succeq
\Lambda_{\mathcal D}(j)^n.
\]
\end{theorem}

\begin{proof}
Fix $j\in\mathfrak r_{\mathbf M}(\iota)$ and write
$A(\mathbf D):=A_j(\mathbf D)$ and $\lambda:=\Lambda_{\mathcal D}(j)$.

We first show that $R_j$ is forward invariant under every $\mathbf D\in\mathcal D$.
If $a\in R_j$ and $[\mathbf D]_{ab}>0$, choose $\mathbf E\in\mathcal D$ and a path
$j=i_0,i_1,\ldots,i_m=a$ such that $[\mathbf E]_{i_{r-1}i_r}>0$ for all $r\in[m]$.
Replacing the $a$-th row of $\mathbf E$ by the $a$-th row of $\mathbf D$ preserves the path and adds the edge $a\to b$.
Hence $b\in R_j$.

Therefore, for every sequence $\mathbf D_1,\ldots,\mathbf D_n\in\mathcal D$,
\[
\bigg\|
\bbbxi_j\prod_{m=1}^n \mathbf D_m
\bigg\|_1
=
\bigg\|
\bbbxi_j\prod_{m=1}^n A(\mathbf D_m)
\bigg\|_1.
\]

Let $A^*:=A(\mathbf D_j^*)$ and let $\mathbf v_j^*>0$ be its right Perron vector, so
$A^*\mathbf v_j^*=\lambda\mathbf v_j^*$.
We claim that $A(\mathbf D)\mathbf v_j^*\ge \lambda\mathbf v_j^*$ for every $\mathbf D\in\mathcal D$.
Assume not.
Then for some $\mathbf D\in\mathcal D$ and some $a_0\in R_j$ we have
$[A(\mathbf D)\mathbf v_j^*]_{a_0}<\lambda[\mathbf v_j^*]_{a_0}$.
Let $\mathbf D'$ be obtained from $\mathbf D_j^*$ by replacing its $a_0$-th row by the $a_0$-th row of $\mathbf D$.
By assumption, $A(\mathbf D')$ is primitive.
Also $A(\mathbf D')\mathbf v_j^*\le \lambda\mathbf v_j^*$, and the inequality is strict in the coordinate $a_0$.
Choose $q$ with $A(\mathbf D')^q>\mathbf 0$ and set $\mathbf u:=A(\mathbf D')^q\mathbf v_j^*$.
Then $\mathbf u>\mathbf 0$ and $A(\mathbf D')\mathbf u<\lambda\mathbf u$ coordinatewise.
The Collatz--Wielandt formula gives $\rho(A(\mathbf D'))<\lambda$.
Since $A(\mathbf D')$ is primitive, its whole trim is reachable from $j$, so
\[
\max_{\ell\in\mathfrak R_{\mathbf D'}(j)}
\rho\big(\mathbf B_\ell(\mathbf D')\big)
=
\rho\big(A(\mathbf D')\big)
<
\lambda,
\]
contrary to the definition of $\lambda$.
This proves the claim.

Iterating the claim gives
$\prod_{m=1}^n A(\mathbf D_m)\mathbf v_j^*\ge \lambda^n\mathbf v_j^*$.
Hence
\[
\bbbxi_j\prod_{m=1}^n A(\mathbf D_m)\mathbf v_j^*
\ge
\lambda^n[\mathbf v_j^*]_j.
\]
Since $\|\mathbf w\|_1\ge (\mathbf w\mathbf v_j^*)/\|\mathbf v_j^*\|_\infty$ for every non-negative row vector $\mathbf w$, we obtain
\[
\bigg\|
\bbbxi_j\prod_{m=1}^n \mathbf D_m
\bigg\|_1
\ge
\frac{[\mathbf v_j^*]_j}{\|\mathbf v_j^*\|_\infty}\lambda^n.
\]
This proves the lower bound.

For the limit, define the Bellman operator
$\mathcal F_j:\mathbb R_{\ge 0}^{|R_j|}\to\mathbb R_{\ge 0}^{|R_j|}$ by
$[\mathcal F_j(\mathbf x)]_a:=\min_{\mathbf D\in\mathcal D}[A(\mathbf D)\mathbf x]_a$.
Because the rows are chosen independently,
\[
[\mathcal F_j^n(\mathbf 1)]_j
=
\min_{\mathbf D_1,\ldots,\mathbf D_n\in\mathcal D}
\bigg\|
\bbbxi_j\prod_{m=1}^n \mathbf D_m
\bigg\|_1.
\]
The claim above implies $\mathcal F_j(\mathbf v_j^*)\ge \lambda\mathbf v_j^*$.
Since $\mathbf D_j^*\in\mathcal D$, we also have
$\mathcal F_j(\mathbf v_j^*)\le A^*\mathbf v_j^*=\lambda\mathbf v_j^*$.
Thus $\mathcal F_j(\mathbf v_j^*)=\lambda\mathbf v_j^*$.

Set $\mathcal G_j:=\lambda^{-1}\mathcal F_j$.
Choose $a_0,b_0>0$ so that $a_0\mathbf v_j^*\le \mathbf 1\le b_0\mathbf v_j^*$.
Then $a_0\mathbf v_j^*\le \mathcal G_j^n(\mathbf 1)\le b_0\mathbf v_j^*$ for all $n$.
Write $\mathbf y_n:=\mathcal G_j^n(\mathbf 1)$ and
$\beta_n:=\max_a[\mathbf y_n]_a/[\mathbf v_j^*]_a$.
Then $(\beta_n)$ is decreasing, hence converges.

Let $\mathbf y$ be an accumulation point of $(\mathbf y_n)$.
If $\mathbf y\ne \beta\mathbf v_j^*$, where $\beta:=\lim\beta_n$, then
$\mathbf y<\beta\mathbf v_j^*$ in at least one coordinate.
Choose $q$ with $(A^*)^q>\mathbf 0$.
Since $\mathcal F_j(\mathbf x)\le A^*\mathbf x$ for every $\mathbf x\ge \mathbf 0$, we get
\[
\mathcal G_j^q(\mathbf y)
\le
\lambda^{-q}(A^*)^q\mathbf y
<
\beta\lambda^{-q}(A^*)^q\mathbf v_j^*
=
\beta\mathbf v_j^*,
\]
which contradicts the definition of $\beta$.
Hence every accumulation point of $(\mathbf y_n)$ equals $\beta\mathbf v_j^*$, so
$\mathbf y_n\to \beta\mathbf v_j^*$.

Therefore
\[
\frac{1}{\lambda^n}
\min_{\mathbf D_1,\ldots,\mathbf D_n\in\mathcal D}
\bigg\|
\bbbxi_j\prod_{m=1}^n \mathbf D_m
\bigg\|_1
=
[\mathbf y_n]_j
\to
\beta[\mathbf v_j^*]_j.
\]
Since $\mathbf v_j^*>0$, the limit is strictly positive.
\end{proof}

\begin{lemma}\label{lemma:diameter}
For every colour $j\in\mathfrak r_{\mathbf M}(\iota)$,
\[
{ \dist_{\Xi_j^n}(\mathfrak v^+,\mathfrak v^-)}
\asymp
\diam(\Xi_j^n)
\asymp
\Lambda_{\mathcal D}(j)^n.
\]
\end{lemma}

\begin{proof}
By Lemma~\ref{lemma:preivous_results}(iv) and Theorem~\ref{theorem:distance},
\[
\frac{\dist_{\Xi_j^n}(\mathfrak v^+,\mathfrak v^-)}{\Lambda_{\mathcal D}(j)^n}
\to c_j
\]
for some $c_j\in\mathbb R_+$.
Hence $\dist_{\Xi_j^n}(\mathfrak v^+,\mathfrak v^-)\asymp \Lambda_{\mathcal D}(j)^n$.

Every admissible path between the planting vertices has length at least $2$, so every row of $\mathbf D_j^*$ has row sum at least $2$.
The forward invariance of $R_j$ shows that no positive entry is lost when we pass to the trim.
Thus every row of $A_j(\mathbf D_j^*)$ still has row sum at least $2$, and the Collatz--Wielandt formula gives
$\Lambda_{\mathcal D}(j)\ge 2$.
Since $\sum_{m=1}^n \Lambda_{\mathcal D}(j)^m\asymp \Lambda_{\mathcal D}(j)^n$, Lemma~\ref{lemma:preivous_results}(v) yields
\[
\diam(\Xi_j^n)
\asymp
\sum_{m=1}^n \dist_{\Xi_j^m}(\mathfrak v^+,\mathfrak v^-)
\asymp
\Lambda_{\mathcal D}(j)^n.
\]
\end{proof}

\subsection{Scaling limit}

We now pass from distance growth on finite graphs to the limit metric space.
The key observation is that not every descendant tile remains visible after normalisation by the diameter.
A tile survives precisely when its colour stays in the same distance layer as the initial colour.
This gives a canonical decomposition of the limit into a visible part built from surviving tiles and a collapsed part formed by points.
The latter will be negligible for Hausdorff dimension, but it must still be isolated carefully.

For each colour $j\in[K]$, set
\[
I_{\mathrm{dist}}(j)
:=
\big\{
a\in\mathfrak r_{\mathbf M}(j):
\Lambda_{\mathcal D}(a)=\Lambda_{\mathcal D}(j)
\big\},
\qquad
\mathbf M_{\mathrm{dist}}(j)
:=
[\mathbf M]_{I_{\mathrm{dist}}(j)\times I_{\mathrm{dist}}(j)}.
\]
A level-$n$ tile in the graph sequence of initial colour $j$ means the descendant cell determined by one edge of $\Xi_j^n$.
We call such a tile \emph{surviving} if its colour belongs to $I_{\mathrm{dist}}(j)$.

\begin{lemma}\label{lem:distance-monotone}
If $[\mathbf M]_{ab}>0$, then $\Lambda_{\mathcal D}(b)\le \Lambda_{\mathcal D}(a)$.
In particular, along every ancestral colour chain, the values of $\Lambda_{\mathcal D}$ are non-increasing.
\end{lemma}

\begin{proof}
If $[\mathbf M]_{ab}>0$, then $\Xi_a^n$ contains a copy of $\Xi_b^{n-1}$ for every $n\ge 1$.
Hence $\diam(\Xi_b^{n-1})\le \diam(\Xi_a^n)$.
Lemma~\ref{lemma:diameter} gives
$\Lambda_{\mathcal D}(b)^{n-1}\preceq \Lambda_{\mathcal D}(a)^n$.
Taking $n$-th roots and letting $n\to\infty$ yields
$\Lambda_{\mathcal D}(b)\le \Lambda_{\mathcal D}(a)$.
\end{proof}

\begin{lemma}\label{lem:surviving-tile-count}
For every colour $j\in[K]$, the number of level-$n$ surviving tiles in the graph sequence of initial colour $j$ is
\[
\big\|
\bbbxi_j\mathbf M_{\mathrm{dist}}(j)^n
\big\|_1.
\]
More precisely, for each $a\in I_{\mathrm{dist}}(j)$, the number of level-$n$ surviving tiles of colour $a$ is
\[
[\bbbxi_j\mathbf M_{\mathrm{dist}}(j)^n]_a.
\]
\end{lemma}

\begin{proof}
By Lemma~\ref{lem:distance-monotone}, a descendant tile is surviving if and only if every colour in its ancestral chain lies in $I_{\mathrm{dist}}(j)$.
Thus, at each substitution step, only the transitions recorded by the restricted matrix $\mathbf M_{\mathrm{dist}}(j)$ remain.
The conclusion follows by induction on $n$.
\end{proof}

\begin{definition}\label{def:GH-scaling-limit}
Let $(G^n)_{n\in\mathbb N}$ be a sequence of finite connected graphs.
For each $n$, every edge of $G^n$ is scaled by the common factor $1/\diam(G^n)$, giving the compact metric space
\[
\big(G^n,\hat\dist_{G^n}\big),
\qquad
\hat\dist_{G^n}:=\frac{\dist_{G^n}}{\diam(G^n)}.
\]
The sequence possesses a Gromov--Hausdorff scaling limit if $(G^n,\hat\dist_{G^n})$ converges in the Gromov--Hausdorff topology.
\end{definition}

\begin{theorem}\label{thm:GH-limit}
For every colour $j\in\mathfrak r_{\mathbf M}(\iota)$, the graph sequence
$(\Xi_j^n,\hat\dist_{\Xi_j^n})_{n\in\mathbb N}$ has a Gromov--Hausdorff scaling limit, denoted by
$(\Xi_j,\hat\dist_{\Xi_j})$.

Moreover, a level-$n$ tile in the graph sequence of initial colour $j$ has positive diameter in
$(\Xi_j,\hat\dist_{\Xi_j})$ if and only if it is surviving.
\end{theorem}

\begin{proof}
Fix $j\in\mathfrak r_{\mathbf M}(\iota)$.
For each $n$, define the auxiliary metric
$\tilde\dist_{\Xi_j^n}:=\Lambda_{\mathcal D}(j)^{-n}\dist_{\Xi_j^n}$.

Fix $k\ge 0$ and $x,y\in V(\Xi_j^k)$.
As in the original argument, it is enough to consider the case where $x$ and $y$ are joined in $\Xi_j^k$ by a single edge of colour $a$.
Then
\[
\tilde\dist_{\Xi_j^n}(x,y)
=
\Lambda_{\mathcal D}(j)^{-n}
\dist_{\Xi_a^{n-k}}(\mathfrak v^+,\mathfrak v^-).
\]
By Theorem~\ref{theorem:distance}, the right-hand side converges to a positive limit if
$\Lambda_{\mathcal D}(a)=\Lambda_{\mathcal D}(j)$ and to $0$ otherwise.
The case $\Lambda_{\mathcal D}(a)>\Lambda_{\mathcal D}(j)$ is excluded by Lemma~\ref{lem:distance-monotone}.
Since there are only finitely many simple paths in $\Xi_j^k$ joining $x$ and $y$, every pairwise distance on
$V(\Xi_j^k)$ converges in the auxiliary scaling.

Taking the metric quotient of $\bigcup_{m\ge 0}V(\Xi_j^m)$ by the resulting pseudo-metric and then completing it, we obtain a compact metric space, because Lemma~\ref{lemma:diameter} gives
$\diam(\Xi_j^n)\asymp \Lambda_{\mathcal D}(j)^n$ and therefore every point of $\Xi_j^n$ lies within distance
$\preceq \Lambda_{\mathcal D}(j)^{-k}$ of $V(\Xi_j^k)$.
The same estimate gives the Gromov--Hausdorff convergence in the auxiliary scaling.

Now diameter is continuous under Gromov--Hausdorff convergence, and
$\Lambda_{\mathcal D}(j)^{-n}\diam(\Xi_j^n)$ converges to a positive limit by Lemma~\ref{lemma:diameter}.
Rescaling back from the auxiliary metric to $\hat\dist_{\Xi_j^n}$ gives the Gromov--Hausdorff limit
$(\Xi_j,\hat\dist_{\Xi_j})$.

Finally, let $T$ be a level-$n$ tile in the graph sequence of initial colour $j$, and let $a$ be its colour.
By the argument above, the diameter of $T$ in the limit is positive if and only if
$\Lambda_{\mathcal D}(a)=\Lambda_{\mathcal D}(j)$.
By Lemma~\ref{lem:distance-monotone}, this is equivalent to saying that every colour in the ancestral chain of $T$
has the same distance rate as $j$, which is exactly the condition that $T$ is surviving.
\end{proof}

For each colour $j\in\mathfrak r_{\mathbf M}(\iota)$, let $C_j$ be the set of points of $\Xi_j$ that belong to only finitely many surviving tiles.

\begin{lemma}\label{lem:collapsed-countable}
For every colour $j\in\mathfrak r_{\mathbf M}(\iota)$, the set $C_j$ is countable.
\end{lemma}

\begin{proof}
Fix $x\in C_j$.
Since $x$ belongs to only finitely many surviving tiles, there is a largest level $m$ for which $x$ belongs to a surviving level-$m$ tile.
Every level-$(m+1)$ tile containing $x$ is then non-surviving, and by Theorem~\ref{thm:GH-limit} each such tile collapses to a single point.
Thus every point of $C_j$ is the image of some non-surviving tile.
There are only countably many tiles altogether, hence $C_j$ is countable.
\end{proof}

\begin{lemma}\label{lem:surviving-basis}
If $x\in \Xi_j\setminus C_j$, then for every open neighbourhood $U$ of $x$ in $(\Xi_j,\hat\dist_{\Xi_j})$,
there exists a surviving tile $T$ such that $x\in T\subset U$.
\end{lemma}

\begin{proof}
Since $x\notin C_j$, it belongs to infinitely many surviving tiles.
By Theorem~\ref{thm:GH-limit}, every surviving tile of colour $a\in I_{\mathrm{dist}}(j)$ has diameter
$\asymp \Lambda_{\mathcal D}(a)^{-n}=\Lambda_{\mathcal D}(j)^{-n}$ at level $n$.
Because the colour set is finite, the implied constants are uniform over $a\in I_{\mathrm{dist}}(j)$.
Hence the diameters of surviving tiles containing $x$ tend to $0$.
Choosing one of them small enough gives $x\in T\subset U$.
\end{proof}

\subsection{Hausdorff dimension}

We next extract the global metric dimension of the limit.
At this stage the relevant objects are the surviving tiles and the mass growth they carry.
Unlike the degree dimension in the next section, the Hausdorff dimension only records the largest asymptotic growth that remains visible in the metric space.
This is why the answer is governed by a single surviving mass rate.
The proof follows the standard upper and lower bound strategy: cover by surviving tiles, and then place disjoint balls inside a suitably chosen primitive subfamily of tiles.

Throughout this subsection, we additionally assume that $\mathbf M_{\mathrm{dist}}(j)$ is primitive-Frobenius for every colour
$j\in\mathfrak r_{\mathbf M}(\iota)$.
For each such $j$, define
\[
\Lambda_{\mathbf M}^{\mathrm{surv}}(j)
:=
\max_{k\in \mathfrak R_{\mathbf M_{\mathrm{dist}}(j)}(j)}
\rho\big(\mathbf B_k(\mathbf M_{\mathrm{dist}}(j))\big).
\]

\begin{theorem}\label{thm:singlefractal}
For every colour $j\in\mathfrak r_{\mathbf M}(\iota)$,
\[
\dimh(\Xi_j)
=
\frac{
\log \Lambda_{\mathbf M}^{\mathrm{surv}}(j)
}{
\log \Lambda_{\mathcal D}(j)
}.
\]
\end{theorem}

\begin{proof}
Fix $j\in\mathfrak r_{\mathbf M}(\iota)$ and write
\[
N_n(j):=
\big\|
\bbbxi_j\mathbf M_{\mathrm{dist}}(j)^n
\big\|_1.
\]
By Lemma~\ref{lem:surviving-tile-count}, $N_n(j)$ is the number of level-$n$ surviving tiles in the graph sequence of initial colour $j$.
Lemma~\ref{lemma:asymp-value} applied to $\mathbf M_{\mathrm{dist}}(j)$ gives
$N_n(j)\asymp n^{q_j-1}\Lambda_{\mathbf M}^{\mathrm{surv}}(j)^n$ for some integer $q_j\ge 1$.

We first prove the upper bound.
By Theorem~\ref{thm:GH-limit}, $\Xi_j\setminus C_j$ is covered by the level-$n$ surviving tiles, and
$C_j$ is countable by Lemma~\ref{lem:collapsed-countable}.
Lemma~\ref{lemma:diameter} and Theorem~\ref{thm:GH-limit} show that every level-$n$ surviving tile has diameter
$\preceq \Lambda_{\mathcal D}(j)^{-n}$.
Therefore, for every $s>0$,
\[
\mathcal H^s_{\delta_n}(\Xi_j)
\preceq
N_n(j)\Lambda_{\mathcal D}(j)^{-ns}
\preceq
n^{q_j-1}\Lambda_{\mathbf M}^{\mathrm{surv}}(j)^n\Lambda_{\mathcal D}(j)^{-ns},
\]
with $\delta_n\asymp \Lambda_{\mathcal D}(j)^{-n}$.
If $s>\log\Lambda_{\mathbf M}^{\mathrm{surv}}(j)/\log\Lambda_{\mathcal D}(j)$, the right-hand side tends to $0$.
Hence
\[
\dimh(\Xi_j)
\le
\frac{\log \Lambda_{\mathbf M}^{\mathrm{surv}}(j)}{\log \Lambda_{\mathcal D}(j)}.
\]

We now prove the lower bound.
Choose a primitive block $\mathbf B$ of $\mathbf M_{\mathrm{dist}}(j)$ such that
$\rho(\mathbf B)=\Lambda_{\mathbf M}^{\mathrm{surv}}(j)$ and $\mathbf B$ is reachable from the block containing $j$.
Since $\mathbf B$ is reachable from $j$, there exists a surviving tile $T_0\subset \Xi_j$ whose colour belongs to the index set of $\mathbf B$.
For each $n\ge 1$, let $\mathcal T_n$ be the family of descendants of $T_0$ obtained by following only colours in the block $\mathbf B$ for the next $n$ steps.
Because $\mathbf B$ is primitive,
\[
|\mathcal T_n|
\asymp
\rho(\mathbf B)^n
=
\Lambda_{\mathbf M}^{\mathrm{surv}}(j)^n.
\]

Every tile $T\in\mathcal T_n$ is connected and has two boundary points, say $p_T^+$ and $p_T^-$.
By Lemma~\ref{lemma:diameter} and Theorem~\ref{thm:GH-limit}, there exists a constant $c_0>0$, independent of $T$ and $n$, such that
\[
\hat\dist_{\Xi_j}(p_T^+,p_T^-)\ge c_0\,\Lambda_{\mathcal D}(j)^{-n}.
\]
Since $T$ is connected, there exists $x_T\in T$ with
\[
\min\big\{
\hat\dist_{\Xi_j}(x_T,p_T^+),
\hat\dist_{\Xi_j}(x_T,p_T^-)
\big\}
\ge
\frac{c_0}{2}\Lambda_{\mathcal D}(j)^{-n}.
\]
Indeed, otherwise $T$ would be contained in the disjoint union of the two open balls
$B(p_T^+,c_0\Lambda_{\mathcal D}(j)^{-n}/2)$ and
$B(p_T^-,c_0\Lambda_{\mathcal D}(j)^{-n}/2)$, contradicting connectedness.
Thus each $T\in\mathcal T_n$ contains a ball of radius
$c_0\Lambda_{\mathcal D}(j)^{-n}/4$.
These balls are pairwise disjoint, because distinct tiles of the same level meet only along their boundary images.

Hence, for every $s>0$,
\[
\mathcal H^s_{\delta_n}(\Xi_j)
\succeq
|\mathcal T_n|\,\Lambda_{\mathcal D}(j)^{-ns}
\asymp
\Lambda_{\mathbf M}^{\mathrm{surv}}(j)^n\Lambda_{\mathcal D}(j)^{-ns},
\]
with $\delta_n\asymp \Lambda_{\mathcal D}(j)^{-n}$.
If $s<\log\Lambda_{\mathbf M}^{\mathrm{surv}}(j)/\log\Lambda_{\mathcal D}(j)$, the right-hand side diverges.
Therefore
\[
\dimh(\Xi_j)
\ge
\frac{\log \Lambda_{\mathbf M}^{\mathrm{surv}}(j)}{\log \Lambda_{\mathcal D}(j)}.
\]
Combining the two bounds gives the result.
\end{proof}

\begin{lemma}\label{lem:tile-dimension}
Let $T\subset \Xi_\iota$ be a surviving tile of colour $j\in I_{\mathrm{dist}}(\iota)$.
Then
\[
\dimh(T)=\dimh(\Xi_j).
\]
\end{lemma}

\begin{proof}
Inside $T$, the surviving descendants of depth $n$ are counted by
$\|\bbbxi_j\mathbf M_{\mathrm{dist}}(j)^n\|_1$, exactly as in Lemma~\ref{lem:surviving-tile-count}.
By Theorem~\ref{thm:GH-limit}, their diameters are all $\asymp \diam(T)\Lambda_{\mathcal D}(j)^{-n}$.
Repeating the upper and lower bound arguments from the proof of Theorem~\ref{thm:singlefractal} inside $T$ therefore gives
\[
\dimh(T)
=
\frac{
\log \Lambda_{\mathbf M}^{\mathrm{surv}}(j)
}{
\log \Lambda_{\mathcal D}(j)
}
=
\dimh(\Xi_j).
\]
\end{proof}

\subsection{Coarse fractal spectrum}

The final step of the metric analysis is local rather than global.
Here the central object is the ball $B(x,r)$ inside the limit metric space.
The collapsed part contributes only countably many points, so it cannot create new Hausdorff dimensions.
What remains is a union of surviving tiles, and each surviving tile carries the Hausdorff dimension of its colour.
This is why the coarse spectrum is finite and discrete.

For a compact metric space $(X,d)$, write
\[
\mathscr S(X,d)
:=
\Big\{
\dimh\big(B_d(x,r)\big):
x\in X,\ r\in(0,1)
\Big\}.
\]
We say that $(X,d)$ is a multifractal with a finite discrete spectrum if
$1<|\mathscr S(X,d)|<\infty$.
Also set
\[
\mathscr M_{\mathrm{surv}}(\iota)
:=
\big\{
\Lambda_{\mathbf M}^{\mathrm{surv}}(j):
j\in I_{\mathrm{dist}}(\iota)
\big\}.
\]
We say that $(\Xi_\iota,\hat\dist_{\Xi_\iota})$ satisfies the balanced-distance and divergent-mass condition if
$1<|\mathscr M_{\mathrm{surv}}(\iota)|<\infty$.

\begin{lemma}\label{lem:ball-dimension}
For every ball $B(x,r)$ in $(\Xi_\iota,\hat\dist_{\Xi_\iota})$,
\[
\dimh\big(B(x,r)\big)
=
\sup\big\{
\dimh(T):
T\subset B(x,r),\ 
T \text{ is a surviving tile in } \Xi_\iota
\big\}.
\]
\end{lemma}

\begin{proof}
Let
\[
s(B(x,r))
:=
\sup\big\{
\dimh(T):
T\subset B(x,r),\ 
T \text{ is a surviving tile in } \Xi_\iota
\big\}.
\]
Since every such tile is a subset of $B(x,r)$, we clearly have $s(B(x,r))\le \dimh(B(x,r))$.

Set $E:=B(x,r)\setminus C_\iota$.
For every $y\in E$, Lemma~\ref{lem:surviving-basis} gives a surviving tile $T_y$ such that
$y\in T_y\subset B(x,r)$.
Hence
\[
B(x,r)
=
\big(B(x,r)\cap C_\iota\big)\cup \bigcup_{y\in E}T_y.
\]
Since $\Xi_\iota$ is second countable, there is a countable subcover
$E\subset \bigcup_{m=1}^\infty T_m$ with each $T_m\subset B(x,r)$ surviving.
By Lemma~\ref{lem:collapsed-countable}, the set $B(x,r)\cap C_\iota$ is countable, hence has Hausdorff dimension $0$.
Therefore
\[
\dimh\big(B(x,r)\big)
=
\dimh\Big(\bigcup_{m=1}^\infty T_m\Big)
=
\sup_{m\ge 1}\dimh(T_m)
\le
s(B(x,r)).
\]
This proves the equality.
\end{proof}

\begin{theorem}\label{thm:multifractal}
For the fixed initial colour $\iota$,
\[
\mathscr S(\Xi_\iota,\hat\dist_{\Xi_\iota})
=
\big\{
\dimh(\Xi_j):
j\in I_{\mathrm{dist}}(\iota)
\big\}
=
\left\{
\frac{
\log \Lambda_{\mathbf M}^{\mathrm{surv}}(j)
}{
\log \Lambda_{\mathcal D}(j)
}
:
j\in I_{\mathrm{dist}}(\iota)
\right\}.
\]
Consequently, $(\Xi_\iota,\hat\dist_{\Xi_\iota})$ satisfies the balanced-distance and divergent-mass condition if and only if it is a multifractal with a finite discrete spectrum, that is,
\[
1<
\big|
\mathscr S(\Xi_\iota,\hat\dist_{\Xi_\iota})
\big|
<\infty.
\]
\end{theorem}

\begin{proof}
Let
\[
S_0
:=
\big\{
\dimh(\Xi_j):
j\in I_{\mathrm{dist}}(\iota)
\big\}.
\]
The set $S_0$ is finite because the colour set is finite.

Fix a ball $B(x,r)$ in $(\Xi_\iota,\hat\dist_{\Xi_\iota})$.
By Lemma~\ref{lem:ball-dimension},
\[
\dimh\big(B(x,r)\big)
=
\sup\big\{
\dimh(T):
T\subset B(x,r),\ 
T \text{ is a surviving tile}
\big\}.
\]
Every surviving tile has colour in $I_{\mathrm{dist}}(\iota)$, and Lemma~\ref{lem:tile-dimension} shows that its Hausdorff dimension belongs to $S_0$.
Since $S_0$ is finite, the above supremum is attained.
Thus $\dimh(B(x,r))\in S_0$, so
$\mathscr S(\Xi_\iota,\hat\dist_{\Xi_\iota})\subseteq S_0$.

Conversely, fix $j\in I_{\mathrm{dist}}(\iota)$ and choose a surviving tile $T\subset \Xi_\iota$ of colour $j$.
Let $\partial T$ be its two boundary points.
Then $T\setminus \partial T$ is open in $\Xi_\iota$, and removing finitely many points does not change Hausdorff dimension, so
\[
\dimh(T\setminus \partial T)=\dimh(T)=\dimh(\Xi_j)
\]
by Lemma~\ref{lem:tile-dimension}.
Since $\Xi_\iota$ is second countable, there exist open balls $B_m\subset T\setminus \partial T$ such that
$T\setminus \partial T=\bigcup_{m=1}^\infty B_m$.
Therefore
\[
\dimh(\Xi_j)
=
\dimh(T\setminus \partial T)
=
\sup_{m\ge 1}\dimh(B_m).
\]
Each $\dimh(B_m)$ already lies in the finite set $S_0$, so the supremum is attained by some $m$.
Hence $\dimh(B_m)=\dimh(\Xi_j)$, and therefore $\dimh(\Xi_j)\in \mathscr S(\Xi_\iota,\hat\dist_{\Xi_\iota})$.
This proves $S_0\subseteq \mathscr S(\Xi_\iota,\hat\dist_{\Xi_\iota})$.

Thus
\[
\mathscr S(\Xi_\iota,\hat\dist_{\Xi_\iota})
=
\big\{
\dimh(\Xi_j):
j\in I_{\mathrm{dist}}(\iota)
\big\}.
\]
The explicit formula now follows from Theorem~\ref{thm:singlefractal}.

Finally, Theorem~\ref{thm:singlefractal} also shows that
\[
\big|
\mathscr S(\Xi_\iota,\hat\dist_{\Xi_\iota})
\big|
=
\big|
\mathscr M_{\mathrm{surv}}(\iota)
\big|,
\]
so the balanced-distance and divergent-mass condition is equivalent to
$1<|\mathscr S(\Xi_\iota,\hat\dist_{\Xi_\iota})|\leq { K}<\infty$.
\end{proof}

\begin{example}
Consider the reducible EIGS in Figure~\ref{fig:splendor}, with initial colour $\iota=1$.
The mass matrix is
\[
\mathbf M=
\begin{pmatrix}
2&1&1\\
0&4&0\\
0&0&5
\end{pmatrix}.
\]
A direct computation gives
\[
\Lambda_{\mathcal D}(1)=\Lambda_{\mathcal D}(2)=\Lambda_{\mathcal D}(3)=2.
\]
Hence
\[
I_{\mathrm{dist}}(1)=\{1,2,3\},
\qquad
\mathbf M_{\mathrm{dist}}(1)=\mathbf M.
\]
Moreover,
\[
\Lambda_{\mathbf M}^{\mathrm{surv}}(1)=5,\qquad
\Lambda_{\mathbf M}^{\mathrm{surv}}(2)=4,\qquad
\Lambda_{\mathbf M}^{\mathrm{surv}}(3)=5.
\]
Therefore
\[
\mathscr M_{\mathrm{surv}}(1)=\{4,5\},
\]
and Theorem~\ref{thm:multifractal} yields
\[
{ 
\mathscr S(\Xi_1,\hat\dist_{\Xi_1})}
=
\left\{
2,
\frac{\log 5}{\log 2}
\right\}.
\]
In particular, $(\Xi_1,\hat\dist_{\Xi_1})$ satisfies the balanced-distance and divergent-mass condition and is a multifractal with a finite discrete spectrum.
\end{example}

\section{Multiscale-freeness and degree spectrum}

In this section, we introduce the degree scaling limit, an analogue of the Gromov--Hausdorff limit under which vertices with slower degree growth lose their degree in the limit.
Within this framework, we identify a condition called BEDM, which is equivalent to the occurrence of multiscale-freeness in $\Xi_\iota$.
Unlike the Hausdorff dimension studied in the previous section, degree dimension is sensitive to the distribution of degrees and therefore captures finer information about the graph.
The key objects are no longer tiles in a metric space, but exact positive degree levels in the degree scaling limit.
For this reason, the analysis cannot stop at the largest growth rate.
Instead, we must first identify the dominant birth types, then group the exact degree levels they generate into degree classes, and finally count the vertices in each class separately.
Once this class counting is established, both scale-freeness and the degree spectrum follow in a direct way.

\subsection{Degree scaling limit}

We first determine which vertices remain visible in the degree scaling limit.
The correct starting point is the birth type of a vertex, because the local degree growth is propagated by the degree matrix $\mathbf N$.
Two pieces of information matter here: the exponential rate of growth and the polynomial correction.
Only the birth types attaining the dominant pair survive with positive degree in the limit.
The theorem below also records the exact positive degree levels contributed by such surviving birth types.

\begin{definition}\label{def:degree-limit}
Let $\Delta(G):=\max_{w \in V(G)}\deg_G(w)$.
The {\em scaling degree} of a vertex $v\in V(G)$ is $\hat{\deg}_G(v):=\deg_G(v)/\Delta(G)$.
We say that a graph sequence $(G^n)_{n\in\mathbb N}$ has a {\em degree scaling limit} $(G^\infty,\hat{\deg}_{G^\infty})$ if, for every vertex $v\in \bigcup_{n\ge 0}V(G^n)$, the limit
\[
\hat{\deg}_{G^\infty}(v)
:=
\lim_{n\to\infty}\hat{\deg}_{G^n}(v)
\]
exists.
\end{definition}

\begin{notation}\label{not:degree-notation}
Fix the initial colour $\iota\in[K]$.
Define the set of birth types by
\[
\mathcal U_\iota
:=
\{\bkappa(v):v\in V(\Xi_\iota^0)\}
\cup
\bigcup_{j\in\mathfrak r_{\mathbf M}(\iota)}
\{\bkappa(w):w\in V(R_j)\setminus\{\beta_j^+,\beta_j^-\}\}.
\]
For each non-zero $\mathbf u\in\mathcal U_\iota$, define
\[
\lambda(\mathbf u)
:=
\max_{a:[\mathbf u]_a>0}
\max_{\ell\in\mathfrak R_{\mathbf N}(a)}
\rho\big(\mathbf B_\ell(\mathbf N)\big)
\]
and
\[
\tau(\mathbf u)
:=
\max\Big\{
\varkappa_{\mathbf N}(a):
[\mathbf u]_a>0,\
\max_{\ell\in\mathfrak R_{\mathbf N}(a)}
\rho\big(\mathbf B_\ell(\mathbf N)\big)
=
\lambda(\mathbf u)
\Big\}.
\]
Also set
\[
\Lambda_{\mathcal U_\iota}
:=
\max_{\mathbf u\in\mathcal U_\iota}\lambda(\mathbf u)
\qquad\text{and}\qquad
\tau_{\deg}(\iota)
:=
\max\big\{
\tau(\mathbf u):
\mathbf u\in\mathcal U_\iota,\
\lambda(\mathbf u)=\Lambda_{\mathcal U_\iota}
\big\}.
\]
\end{notation}

\begin{lemma}\label{lem:degree-row-asymp}
For every non-zero $\mathbf u\in\mathcal U_\iota$, there exists a constant $c_{\deg}(\mathbf u)\in\mathbb R_+$ such that
\[
\|\mathbf u\mathbf N^n\|_1
\sim
c_{\deg}(\mathbf u)\,
n^{\tau(\mathbf u)-1}\lambda(\mathbf u)^n
\]
as $n\to\infty$.
\end{lemma}

\begin{proof}
Write $\mathbf u=\sum_{a:[\mathbf u]_a>0}[\mathbf u]_a\bbbxi_a$.
Then $\|\mathbf u\mathbf N^n\|_1=\mathbf u\mathbf N^n\mathbf 1$ is a scalar entry of a fixed non-negative row vector multiplied by $\mathbf N^n$ and then by a fixed positive column vector.
Replacing $\mathbf N$ by the principal submatrix indexed by the union of the reachable sets from the support of $\mathbf u$, we may assume that every diagonal block is reachable from that support.
Since $\mathbf N$ is primitive-Frobenius, the only eigenvalue of modulus $\lambda(\mathbf u)$ is the positive real number $\lambda(\mathbf u)$.
The Jordan decomposition of this principal submatrix therefore yields
\[
\|\mathbf u\mathbf N^n\|_1
=
c\,n^d\lambda(\mathbf u)^n
+
o\big(n^d\lambda(\mathbf u)^n\big)
\]
for some $d\ge 0$ and some $c\ne 0$.
Corollary~\ref{cor:row-vector-asymp} shows that the same quantity is asymptotic to
$n^{\tau(\mathbf u)-1}\lambda(\mathbf u)^n$ up to multiplicative constants.
Hence $d=\tau(\mathbf u)-1$.
Since $\|\mathbf u\mathbf N^n\|_1\ge 0$ for all $n$ and is not eventually zero, the leading coefficient must satisfy $c>0$.
This proves the result.
\end{proof}

\begin{theorem}\label{thm:degree-scaling-limit}
If $\Lambda_{\mathcal U_\iota}>1$, then the degree scaling limit $(\Xi_\iota,\hat{\deg}_{\Xi_\iota})$ exists.
Moreover, there exists a constant $C_{\deg}(\iota)\in\mathbb R_+$ such that
\[
\Delta(\Xi_\iota^n)
\sim
C_{\deg}(\iota)\,
n^{\tau_{\deg}(\iota)-1}\Lambda_{\mathcal U_\iota}^n
\]
as $n\to\infty$.
If a vertex $v$ is born in generation $m(v)$ and has birth type $\mathbf u(v)$, then
\[
\hat{\deg}_{\Xi_\iota}(v)
=
\begin{cases}
\alpha(\mathbf u(v))\,\Lambda_{\mathcal U_\iota}^{-m(v)},
&
\lambda(\mathbf u(v))=\Lambda_{\mathcal U_\iota}
\text{ and }
\tau(\mathbf u(v))=\tau_{\deg}(\iota),
\\[1ex]
0,
&
\text{otherwise},
\end{cases}
\]
where $\alpha(\mathbf u):=c_{\deg}(\mathbf u)/C_{\deg}(\iota)$.
\end{theorem}

\begin{proof}
For each birth type $\mathbf u\in\mathcal U_\iota$, let $\mu(\mathbf u)$ be the first generation in which a vertex of type $\mathbf u$ appears.
Set
\[
C_{\deg}(\iota)
:=
\max\big\{
c_{\deg}(\mathbf u)\Lambda_{\mathcal U_\iota}^{-\mu(\mathbf u)}:
\mathbf u\in\mathcal U_\iota,\
\lambda(\mathbf u)=\Lambda_{\mathcal U_\iota},\
\tau(\mathbf u)=\tau_{\deg}(\iota)
\big\}.
\]

We first prove the asymptotic formula for $\Delta(\Xi_\iota^n)$.
If $w\in V(\Xi_\iota^n)$ is born in generation $m\le n$ and has birth type $\mathbf u$, then Lemma~\ref{lemma:preivous_results}(i) gives
$\deg_{\Xi_\iota^n}(w)=\|\mathbf u\mathbf N^{n-m}\|_1$.
By Lemma~\ref{lem:degree-row-asymp}, this equals
\[
c_{\deg}(\mathbf u)\,(n-m)^{\tau(\mathbf u)-1}\lambda(\mathbf u)^{n-m}
\big(1+o(1)\big).
\]

If $\lambda(\mathbf u)<\Lambda_{\mathcal U_\iota}$, then
\[
\sup_{0\le m\le n}
\frac{\lambda(\mathbf u)^{n-m}}{\Lambda_{\mathcal U_\iota}^n}
=
\Lambda_{\mathcal U_\iota}^{-n}
\max_{0\le t\le n}\lambda(\mathbf u)^t
\to 0.
\]
Hence these vertices contribute
$o\big(n^{\tau_{\deg}(\iota)-1}\Lambda_{\mathcal U_\iota}^n\big)$ uniformly in $m\le n$.
If $\lambda(\mathbf u)=\Lambda_{\mathcal U_\iota}$ but $\tau(\mathbf u)<\tau_{\deg}(\iota)$, then
\[
\sup_{0\le m\le n}
\frac{(n-m)^{\tau(\mathbf u)-1}}{n^{\tau_{\deg}(\iota)-1}}
\le
n^{\tau(\mathbf u)-\tau_{\deg}(\iota)}
\to 0,
\]
so these vertices also contribute
$o\big(n^{\tau_{\deg}(\iota)-1}\Lambda_{\mathcal U_\iota}^n\big)$ uniformly in $m\le n$.

Now let $\mathbf u$ satisfy
$\lambda(\mathbf u)=\Lambda_{\mathcal U_\iota}$ and $\tau(\mathbf u)=\tau_{\deg}(\iota)$.
Then for each fixed generation $m$,
\[
\deg_{\Xi_\iota^n}(w)
\sim
c_{\deg}(\mathbf u)\,\Lambda_{\mathcal U_\iota}^{-m}
n^{\tau_{\deg}(\iota)-1}\Lambda_{\mathcal U_\iota}^n.
\]
Since $\Lambda_{\mathcal U_\iota}>1$, vertices born in large generations carry an extra factor
$\Lambda_{\mathcal U_\iota}^{-m}$ and therefore cannot realise the maximum degree.
Hence only finitely many birth generations matter asymptotically, and among them the largest limit constant is attained at the first occurrence generation $\mu(\mathbf u)$.
Because $\mathcal U_\iota$ is finite, we obtain
\[
\Delta(\Xi_\iota^n)
\sim
C_{\deg}(\iota)\,
n^{\tau_{\deg}(\iota)-1}\Lambda_{\mathcal U_\iota}^n.
\]

Finally, let $v$ be born in generation $m(v)$ and have birth type $\mathbf u(v)$.
Dividing the asymptotic formula for $\deg_{\Xi_\iota^n}(v)$ by the asymptotic formula for
$\Delta(\Xi_\iota^n)$ gives
\[
\frac{\deg_{\Xi_\iota^n}(v)}{\Delta(\Xi_\iota^n)}
\to
\begin{cases}
\dfrac{c_{\deg}(\mathbf u(v))}{C_{\deg}(\iota)}
\Lambda_{\mathcal U_\iota}^{-m(v)},
&
\lambda(\mathbf u(v))=\Lambda_{\mathcal U_\iota}
\text{ and }
\tau(\mathbf u(v))=\tau_{\deg}(\iota),
\\[2ex]
0,
&
\text{otherwise}.
\end{cases}
\]
This proves the existence of the degree scaling limit and the displayed formula.
\end{proof}

For each $\mathbf u\in\mathcal U_\iota$, let $\mathbf b_{\mathbf u}\in\mathbb N^{K\times 1}$ be the column vector whose $a$-th entry is the number of interior vertices of birth type $\mathbf u$ in the rule graph $R_a$.
For each integer $m\ge 1$, let $B_{\mathbf u}(m)$ be the number of vertices born in generation $m$ with birth type $\mathbf u$.
Define
\[
I_{\mathrm{deg,surv}}(\iota)
:=
\big\{
\mathbf u\in\mathcal U_\iota\setminus\{\mathbf 0\}:
\lambda(\mathbf u)=\Lambda_{\mathcal U_\iota},\
\tau(\mathbf u)=\tau_{\deg}(\iota),\
B_{\mathbf u}(m)>0 \text{ for arbitrarily large } m
\big\}.
\]
Vertices of dominant type that appear only finitely many times contribute only finitely many positive degree levels, and therefore do not affect the tail $\ell\to 0$.

\begin{lemma}\label{lem:birth-count}
For every birth type $\mathbf u\in\mathcal U_\iota$ and every $m\ge 1$,
\[
B_{\mathbf u}(m)
=
\bbbxi_\iota \mathbf M^{m-1}\mathbf b_{\mathbf u}^{\top}.
\]
\end{lemma}

\begin{proof}
At generation $m$, each edge of colour $a$ in $\Xi_\iota^{m-1}$ is replaced by a copy of the rule graph $R_a$.
That copy contributes exactly $[\mathbf b_{\mathbf u}]_a$ interior vertices of birth type $\mathbf u$.
Hence the total number of such vertices equals the number of colour-$a$ edges in $\Xi_\iota^{m-1}$, weighted by $[\mathbf b_{\mathbf u}]_a$ and summed over all colours $a$.
By Lemma~\ref{lemma:preivous_results}(ii), the row vector counting colour frequencies in $\Xi_\iota^{m-1}$ is $\bbbxi_\iota\mathbf M^{m-1}$.
Therefore
\[
B_{\mathbf u}(m)
=
\bbbxi_\iota \mathbf M^{m-1}\mathbf b_{\mathbf u}^{\top}.
\]
\end{proof}

\begin{lemma}\label{lem:birth-count-asymp}
For every $\mathbf u\in I_{\mathrm{deg,surv}}(\iota)$, there exist an integer $q(\mathbf u)\ge 1$ and a real number
$\Lambda_{\mathbf M}^{\deg}(\mathbf u)>0$ such that
\[
B_{\mathbf u}(m)
\asymp
m^{q(\mathbf u)-1}\big(\Lambda_{\mathbf M}^{\deg}(\mathbf u)\big)^m
\]
as $m\to\infty$.
\end{lemma}

\begin{proof}
Fix $\mathbf u\in I_{\mathrm{deg,surv}}(\iota)$.
By Lemma~\ref{lem:birth-count}, the sequence $B_{\mathbf u}(m)$ is the scalar sequence
$\bbbxi_\iota \mathbf M^{m-1}\mathbf b_{\mathbf u}^{\top}$.
Since $\mathbf M$ is primitive-Frobenius, the Jordan decomposition of the reachable principal submatrix of $\mathbf M$ shows that this scalar is a finite linear combination of terms of the form $m^r\mu^m$.
Hence there exist $d\ge 0$, $\Lambda>0$ and $c\ne 0$ such that
\[
B_{\mathbf u}(m)
=
c\,m^d\Lambda^m+o(m^d\Lambda^m).
\]
Because $\mathbf u\in I_{\mathrm{deg,surv}}(\iota)$, the sequence $B_{\mathbf u}(m)$ is positive for arbitrarily large $m$, so the leading coefficient must satisfy $c>0$.
Setting $q(\mathbf u):=d+1$ and $\Lambda_{\mathbf M}^{\deg}(\mathbf u):=\Lambda$ gives the result.
\end{proof}

\subsection{Scale-freeness}

We now turn from the existence of exact degree levels to their distribution.
The main point is that different dominant birth types may contribute degree values on the same geometric scale, up to a finite shift of the generation index.
This leads to the notion of a degree class.
The correct counting object is therefore not a single birth type, but the whole class of birth types producing the same asymptotic degree scale.
Once the class counting is understood, the scale-free criterion becomes very transparent.

\begin{definition}\label{def:degree-dimension}
For any degree-scaled graph $(G,\hat\deg_G)$ and any positive real number $\ell>0$, define
\[
P_{(G,\hat\deg_G)}(\ell)
:=
\big|\{v\in V(G):\hat{\deg}_G(v)=\ell\}\big|.
\]
We say that the degree scaling limit $(G^\infty,\hat\deg_{G^\infty})$ is \emph{scale-free} if, taking the limit $\ell\to 0$ through all positive values such that
$P_{(G^\infty,\hat\deg_{G^\infty})}(\ell)>0$, the quantity
\[
\dimdeg(G^\infty,\hat\deg_{G^\infty})
=
\lim_{\ell\to 0}
\frac{
\log P_{(G^\infty,\hat\deg_{G^\infty})}(\ell)
}{
-\log \ell
}
\]
exists and is positive.
In that case, this limit is called the \emph{degree dimension} of $(G^\infty,\hat\deg_{G^\infty})$.
\end{definition}

For $\mathbf u,\mathbf v\in I_{\mathrm{deg,surv}}(\iota)$, define
\[
\mathbf u\sim_{\deg}\mathbf v
\quad\Longleftrightarrow\quad
\frac{\alpha(\mathbf u)}{\alpha(\mathbf v)}
\in
\Lambda_{\mathcal U_\iota}^{\mathbb Z}.
\]
Write
\[
\mathscr K_{\deg}(\iota)
:=
I_{\mathrm{deg,surv}}(\iota)/\sim_{\deg}.
\]
For each class $K\in\mathscr K_{\deg}(\iota)$, define
\[
\Lambda_{\mathbf M}^{\deg,\mathrm{eff}}(K)
:=
\max\big\{
\Lambda_{\mathbf M}^{\deg}(\mathbf u):
\mathbf u\in K
\big\}.
\]

\begin{lemma}\label{lem:class-count}
For every class $K\in\mathscr K_{\deg}(\iota)$, there exist a constant $\alpha_K>0$ and an integer $q_K\ge 1$ such that
\[
P_{(\Xi_\iota,\hat{\deg}_{\Xi_\iota})}
\big(\alpha_K\Lambda_{\mathcal U_\iota}^{-m}\big)
\asymp
m^{q_K-1}
\big(\Lambda_{\mathbf M}^{\deg,\mathrm{eff}}(K)\big)^m
\]
as $m\to\infty$.
\end{lemma}

\begin{proof}
Fix a class $K\in\mathscr K_{\deg}(\iota)$ and choose a representative $\mathbf u_K\in K$.
Set $\alpha_K:=\alpha(\mathbf u_K)$.
For each $\mathbf u\in K$, by definition of $\sim_{\deg}$ there exists an integer $s_K(\mathbf u)$ such that
\[
\alpha(\mathbf u)
=
\alpha_K\Lambda_{\mathcal U_\iota}^{s_K(\mathbf u)}.
\]
Hence a vertex of birth type $\mathbf u$ born in generation $m+s_K(\mathbf u)$ has exact degree
\[
\alpha(\mathbf u)\Lambda_{\mathcal U_\iota}^{-(m+s_K(\mathbf u))}
=
\alpha_K\Lambda_{\mathcal U_\iota}^{-m}.
\]
Conversely, every sufficiently small positive exact degree level arises from a dominant birth type, and two different classes cannot produce the same exact degree level.
Therefore, for all sufficiently large $m$,
\[
P_{(\Xi_\iota,\hat{\deg}_{\Xi_\iota})}
\big(\alpha_K\Lambda_{\mathcal U_\iota}^{-m}\big)
=
\sum_{\mathbf u\in K}
B_{\mathbf u}\big(m+s_K(\mathbf u)\big).
\]

By Lemma~\ref{lem:birth-count-asymp}, each summand is asymptotic, up to multiplicative constants, to
\[
m^{q(\mathbf u)-1}
\big(\Lambda_{\mathbf M}^{\deg}(\mathbf u)\big)^m.
\]
Since the sum is finite, its growth is governed by the largest exponential rate, namely
$\Lambda_{\mathbf M}^{\deg,\mathrm{eff}}(K)$.
Among the terms attaining this rate, the largest polynomial exponent determines an integer $q_K\ge 1$.
Thus
\[
P_{(\Xi_\iota,\hat{\deg}_{\Xi_\iota})}
\big(\alpha_K\Lambda_{\mathcal U_\iota}^{-m}\big)
\asymp
m^{q_K-1}
\big(\Lambda_{\mathbf M}^{\deg,\mathrm{eff}}(K)\big)^m.
\]
\end{proof}

\begin{definition}\label{def:BEDM}
We say that $\Xi_\iota$ satisfies the \emph{balanced-degree and divergent-mass condition}
(\emph{BEDM} condition) if there exist
$K_1,K_2\in\mathscr K_{\deg}(\iota)$ such that
\[
\Lambda_{\mathbf M}^{\deg,\mathrm{eff}}(K_1)
\ne
\Lambda_{\mathbf M}^{\deg,\mathrm{eff}}(K_2).
\]
\end{definition}

\begin{theorem}\label{thm:singlescalefree}
Let $\mathscr I=(\Xi_\iota^0,\mathcal R,\mathcal S)$ be a reducible edge iterated graph system, and let
$(\Xi_\iota,\hat{\deg}_{\Xi_\iota})$ be its degree scaling limit.
If $\Lambda_{\mathcal U_\iota}>1$, then $\Xi_\iota$ is scale-free if and only if $\Xi_\iota$ does not satisfy the BEDM condition and
\[
\max_{K\in\mathscr K_{\deg}(\iota)}
\Lambda_{\mathbf M}^{\deg,\mathrm{eff}}(K)
>1.
\]
In that case,
\[
\dimdeg(\Xi_\iota,\hat{\deg}_{\Xi_\iota})
=
\frac{
\log \displaystyle\max_{K\in\mathscr K_{\deg}(\iota)}
\Lambda_{\mathbf M}^{\deg,\mathrm{eff}}(K)
}{
\log \Lambda_{\mathcal U_\iota}
}.
\]
\end{theorem}

\begin{proof}
Set
\[
L_*
:=
\max_{K\in\mathscr K_{\deg}(\iota)}
\Lambda_{\mathbf M}^{\deg,\mathrm{eff}}(K).
\]

Assume first that $\Xi_\iota$ does not satisfy the BEDM condition and that $L_*>1$.
Then every class $K\in\mathscr K_{\deg}(\iota)$ satisfies
$\Lambda_{\mathbf M}^{\deg,\mathrm{eff}}(K)=L_*$.
Let $\ell\to 0$ through positive exact degree levels.
All but finitely many such levels come from classes in $\mathscr K_{\deg}(\iota)$, so for sufficiently small $\ell$ there exist
$K\in\mathscr K_{\deg}(\iota)$ and an integer $m$ such that
\[
\ell=\alpha_K\Lambda_{\mathcal U_\iota}^{-m}.
\]
Since $\mathscr K_{\deg}(\iota)$ is finite,
\[
-\log \ell
=
m\log \Lambda_{\mathcal U_\iota}+O(1).
\]
By Lemma~\ref{lem:class-count},
\[
\log P_{(\Xi_\iota,\hat{\deg}_{\Xi_\iota})}(\ell)
=
m\log L_*+O(\log m).
\]
Therefore
\[
\frac{
\log P_{(\Xi_\iota,\hat{\deg}_{\Xi_\iota})}(\ell)
}{
-\log \ell
}
\to
\frac{\log L_*}{\log \Lambda_{\mathcal U_\iota}}.
\]
Since $L_*>1$, this limit is positive, so $\Xi_\iota$ is scale-free and
\[
\dimdeg(\Xi_\iota,\hat{\deg}_{\Xi_\iota})
=
\frac{\log L_*}{\log \Lambda_{\mathcal U_\iota}}.
\]

Conversely, assume that $\Xi_\iota$ is scale-free.
Then the displayed limit exists and is positive, so $L_*>1$.
If the BEDM condition held, there would exist classes $K_1,K_2\in\mathscr K_{\deg}(\iota)$ such that
\[
\Lambda_{\mathbf M}^{\deg,\mathrm{eff}}(K_1)
\ne
\Lambda_{\mathbf M}^{\deg,\mathrm{eff}}(K_2).
\]
For each $r\in\{1,2\}$, Lemma~\ref{lem:class-count} gives a sequence of exact degree levels
$\alpha_{K_r}\Lambda_{\mathcal U_\iota}^{-m}$ along which
\[
\frac{
\log P_{(\Xi_\iota,\hat{\deg}_{\Xi_\iota})}
\big(\alpha_{K_r}\Lambda_{\mathcal U_\iota}^{-m}\big)
}{
-\log\big(\alpha_{K_r}\Lambda_{\mathcal U_\iota}^{-m}\big)
}
\to
\frac{
\log \Lambda_{\mathbf M}^{\deg,\mathrm{eff}}(K_r)
}{
\log \Lambda_{\mathcal U_\iota}
}.
\]
These two limits are different, contradicting the assumption that a single degree dimension exists.
Hence the BEDM condition fails.
\end{proof}

\begin{figure}[hb]
    \centering
    \newcommand{\drawdiamond}[3]{%
  \begin{scope}[
    shift={(#1)},
    scale=#2,
    rotate=#3
  ]
    \drawarrow{0,0}{2,0.5}{red}
    \drawarrow{0,0}{2,-0.5}{red}
    \drawarrow{2,0.5}{2,-0.5}{red}
    \drawarrow{2,-0.5}{2,0.5}{blue}
  \end{scope}
}

\newcommand{\drawtwoedges}[3]{%
  \begin{scope}[
    shift={(#1)},
    scale=#2,
    rotate=#3
  ]
    \drawarrow{0,0}{2,0}{blue}
    \drawarrow{2,0}{2,0}{blue}
  \end{scope}
}

\resizebox{\textwidth}{!}{%

\begin{tikzpicture}
\begin{scope}[shift={(0, 0)}, scale=0.6]
    \node[draw=none, fill=none, rectangle] at (0,-1) {\bf\large{Rule\,:}};
    
    \drawarrow{2,0}{2,0}{red}
    
    \draw[very thick,>=stealth,->] (5, 0) --++ (1, 0);
    
    \drawarrow{7,0}{2,1}{red}
    \drawarrow{7,0}{2,-1}{red}
    \drawarrow{9,1}{2,-1}{red}
    \drawarrow{9,-1}{2,1}{blue}

    \node[draw=none, fill=none, rectangle] at (7,-0.5) {$\beta^+_1$};
    \node[draw=none, fill=none, rectangle] at (11,-0.5) {$\beta^-_1$};
    \node[draw=none, fill=none, rectangle] at (12,0) {$R_1$};

    \drawarrow{2,0}{2,0}{red}
    
    \draw[very thick,>=stealth,->] (5, 0) --++ (1, 0);

    \drawarrow{2,-2}{2,0}{blue}
    
    \draw[very thick,>=stealth,->] (5, -2) --++ (1, 0);

    \drawarrow{7,-2}{2,0}{blue}
    \drawarrow{9,-2}{2,0}{blue}

    \node[draw=none, fill=none, rectangle] at (7,-2.5) {$\beta^+_2$};
    \node[draw=none, fill=none, rectangle] at (11,-2.5) {$\beta^-_2$};
    \node[draw=none, fill=none, rectangle] at (12,-2) {$R_2$};
    
    \draw[ultra thick, dashed] (-2,-3.5) to (13,-3.5);

\end{scope}
\begin{scope}[shift={(-2, -3)}, scale=0.5]
    \drawarrow{2,0}{2,0}{red}

    \node[draw=none, fill=none, rectangle] at (3,-2) {$\Xi_1^0$};
    
    \draw[very thick,>=stealth,->] (5, 0) --++ (1, 0);
    
    \drawdiamond{7,0}{1}{0}

    \node[draw=none, fill=none, rectangle] at (9,-2) {$\Xi_1^1$};

    \draw[very thick,>=stealth,->] (12, 0) --++ (1, 0);

    \drawdiamond{14,0}{0.5}{30}
    \drawdiamond{14,0}{0.5}{-30}
    \drawdiamond{15.75,1}{0.5}{-30}
    \drawtwoedges{15.75,-1}{0.5}{30}

    \node[draw=none, fill=none, rectangle] at (15.8,-2) {$\Xi_1^2$};

    \draw[very thick,>=stealth,->] (18, 0) --++ (1, 0);

    \drawdiamond{20,0}{0.5}{70}
    \drawdiamond{20,0}{0.5}{20}
    \drawdiamond{20.7,1.85}{0.5}{20}
    \drawtwoedges{21.9,0.67}{0.49}{70}
    \drawdiamond{20,0}{0.5}{-70}
    \drawdiamond{20,0}{0.5}{-20}
    \drawtwoedges{20.7,-1.87}{0.5}{-18}
    \drawdiamond{21.9,-0.65}{0.5}{-70}
    \drawdiamond{22.6,2.55}{0.5}{-20}
    \drawdiamond{22.6,2.55}{0.5}{-70}
    \drawdiamond{24.5, 1.85}{0.5}{-70}
    \drawtwoedges{23.3, 0.65}{0.5}{-20}
    \drawtwoedges{22.65,-2.5}{0.44}{43}
     \drawtwoedges{23.9,-1.3}{0.44}{43}

    \node[draw=none, fill=none, rectangle] at (25,-2) {$\Xi_1^3$};
    
    \draw[very thick,>=stealth,->] (26, 0) --++ (1, 0);

    \node[draw=none, fill=none, rectangle] at (28,0) {\LARGE{...}};
\end{scope}
\end{tikzpicture} 
}
    \caption{The broken diamond hierarchical lattice is scale-free}
    \label{fig:BDHL_2}
\end{figure}
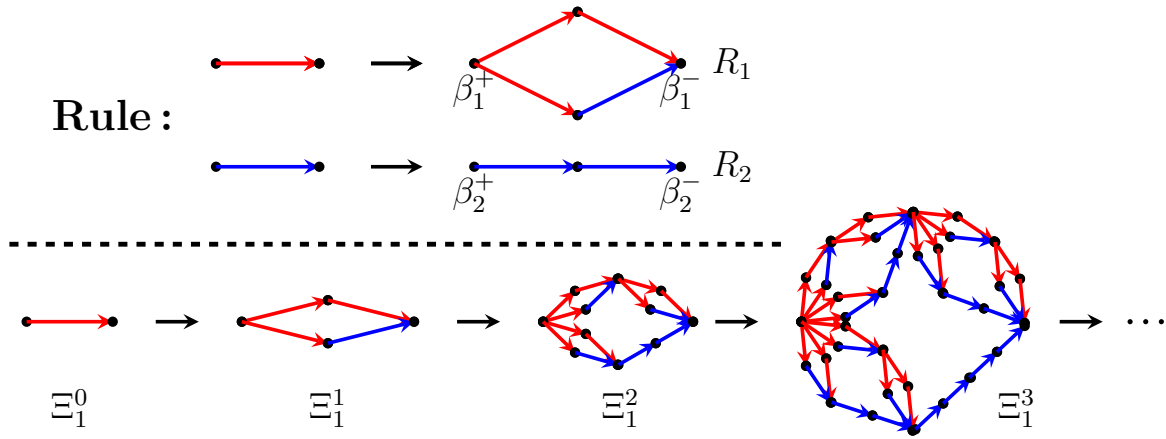

\begin{example}
For the broken diamond hierarchical lattice in Figure~\ref{fig:BDHL_2},
\[
\mathbf M=
\begin{pmatrix}
3 & 1\\[2pt]
0 & 2
\end{pmatrix},
\qquad
\mathbf N=
\begin{pmatrix}
2 & 0 & 0 & 0\\[2pt]
0 & 1 & 0 & 1\\[2pt]
0 & 0 & 1 & 0\\[2pt]
0 & 0 & 0 & 1
\end{pmatrix}.
\]
The matrix $\mathbf M$ has two diagonal blocks with spectral radii $3$ and $2$, while $\mathbf N$ has four diagonal blocks with spectral radii $2,1,1,1$.
The additional blue edge changes the mass growth of the blue rule, but it does not change its degree growth rate.

Since the initial colour is $\iota=1$, we have $\Lambda_{\mathcal U_\iota}=2$.
There is only one surviving degree class, and its effective mass growth rate is $3$.
Hence the BEDM condition fails.
Theorem~\ref{thm:singlescalefree} therefore gives
\[
\dimdeg(\Xi_\iota,\hat{\deg}_{\Xi_\iota})
=
\frac{\log 3}{\log 2}
\approx 1.5850.
\]

Figure~\ref{fig:scale-free-simulation} shows the log--log linear-regression fits of the degree distributions of { 
$\Xi_1^{7}$ and $\Xi_1^{11}$} obtained with Python.
Each plot exhibits a pronounced two-branch structure:
a red branch, which corresponds to the surviving degree layer,
and a blue branch, which corresponds to a suppressed degree class.
As $n\to\infty$, the blue branch is progressively compressed and eventually disappears,
while the slopes ($-1.75$ and $-1.63$) of the surviving red branch approach the theoretical value $-\log 3/\log 2$.

\begin{figure}[ht]
  \centering

  \begin{minipage}[b]{0.49\textwidth}
    \centering
    \includegraphics[width=\linewidth]{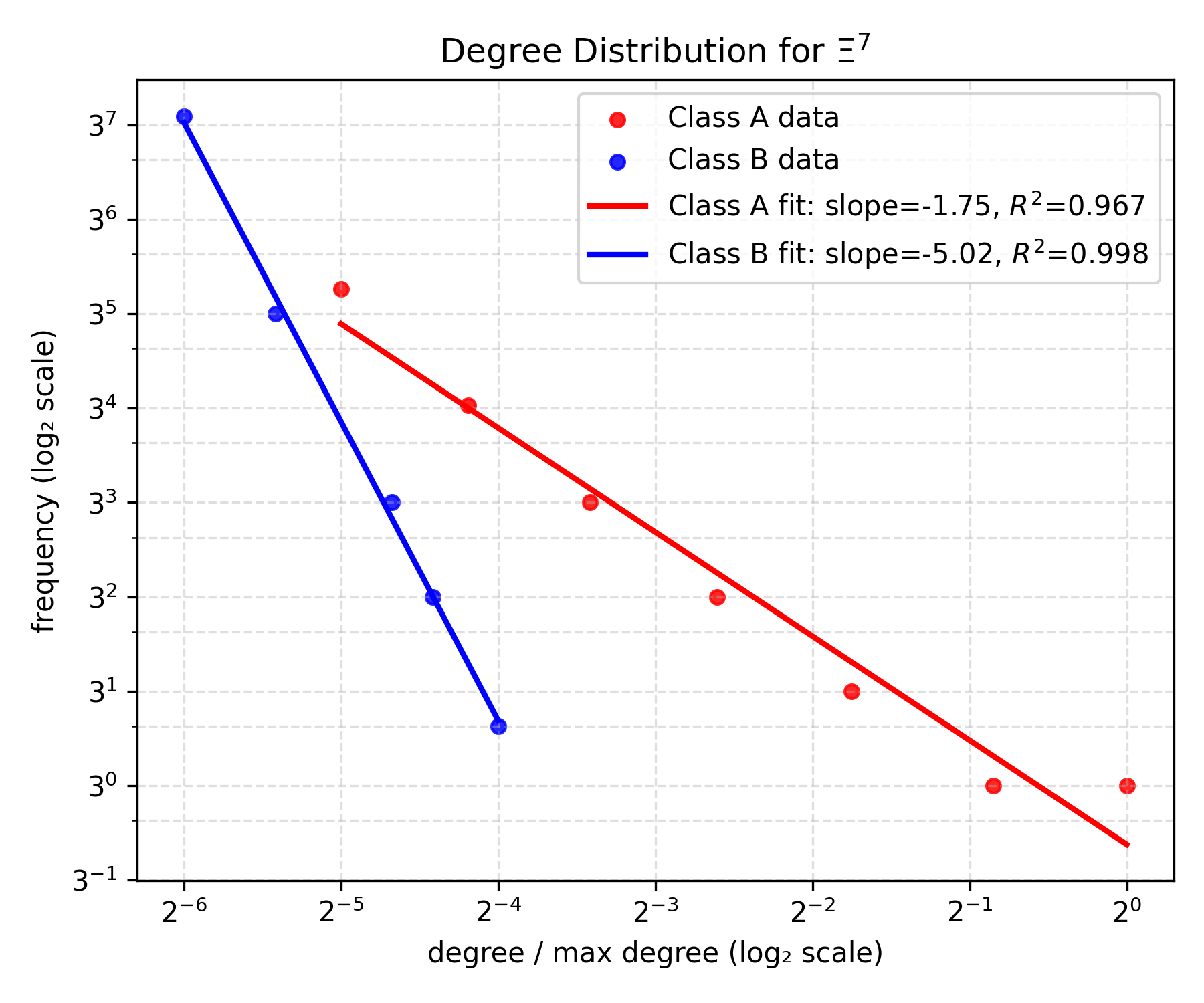}
  \end{minipage}
  \hfill
  \begin{minipage}[b]{0.49\textwidth}
    \centering
    \includegraphics[width=\linewidth]{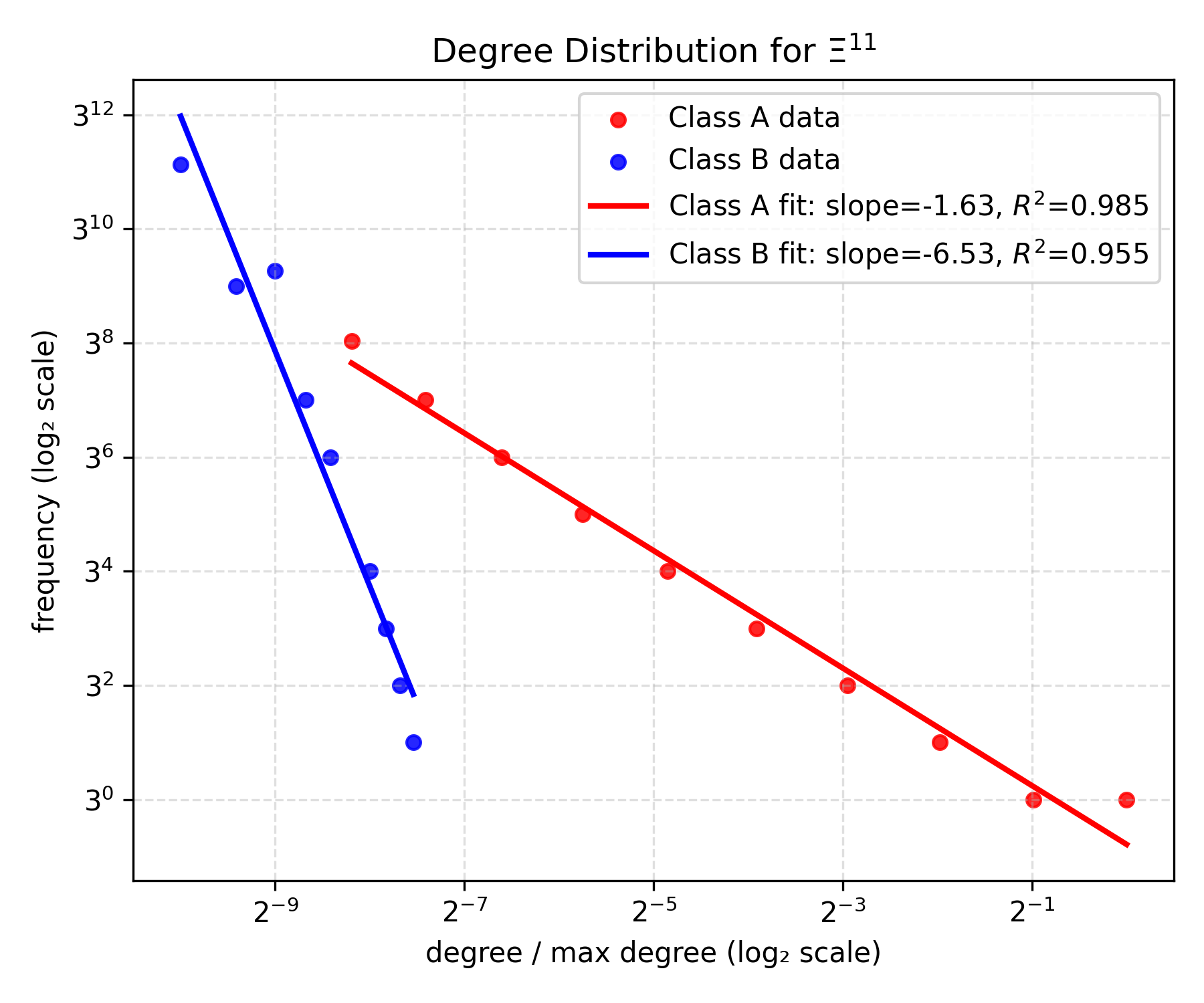}
  \end{minipage}

  \caption{Numerics of the scale-freeness of Broken DHL $\Xi_1^{7}$ and $\Xi_1^{11}$.}
  \label{fig:scale-free-simulation}
\end{figure}
\end{example}

\subsection{Degree spectrum}

We finally record the full tail structure of the exact degree distribution.
Where the previous subsection asked whether all classes share one common exponent, this subsection keeps every class separate.
The resulting degree spectrum is therefore the precise degree-side analogue of the coarse fractal spectrum from the previous section.
Its discreteness comes from the fact that only finitely many degree classes survive.
The only remaining task is to identify the limit point contributed by each class.

\begin{definition}\label{def:degree-spectrum}
Let $(G^\infty,\hat\deg_{G^\infty})$ be a degree-scaling limit and define the
\emph{degree spectrum} by
\[
\mathscr D(G^\infty,\hat\deg_{G^\infty})
:=
\omega\text{-}\lim_{\ell\to 0}
\frac{
\log P_{(G^\infty,\hat\deg_{G^\infty})}(\ell)
}{
-\log \ell
},
\]
where $\omega\text{-}\lim$ denotes the set of all limit points along sequences
$(\ell_m)$ with $\ell_m\downarrow 0$ and
$P_{(G^\infty,\hat\deg_{G^\infty})}(\ell_m)>0$ for all $m$.
We say that $(G^\infty,\hat\deg_{G^\infty})$ is \emph{finite discrete multiscale-free} if
\[
1<\big|\mathscr D(G^\infty,\hat\deg_{G^\infty})\big|<\infty.
\]
\end{definition}

\begin{theorem}\label{thm:multiscalefree}
Let $\mathscr I=(\Xi_\iota^{0},\mathcal R,\mathcal S)$ be a reducible edge iterated graph system with degree scaling limit
$(\Xi_\iota,\hat{\deg}_{\Xi_\iota})$.
If $\Lambda_{\mathcal U_\iota}>1$, then
\[
\mathscr D(\Xi_\iota,\hat{\deg}_{\Xi_\iota})
=
\bigg\{
\frac{
\log \Lambda_{\mathbf M}^{\deg,\mathrm{eff}}(K)
}{
\log \Lambda_{\mathcal U_\iota}
}
:
K\in\mathscr K_{\deg}(\iota)
\bigg\}.
\]
Consequently, $\Xi_\iota$ satisfies the BEDM condition if and only if it is finite discrete multiscale-free.
\end{theorem}

\begin{proof}
Let $K\in\mathscr K_{\deg}(\iota)$.
By Lemma~\ref{lem:class-count}, the exact degree sequence
$\alpha_K\Lambda_{\mathcal U_\iota}^{-m}$ contributes the limit point
\[
\frac{
\log \Lambda_{\mathbf M}^{\deg,\mathrm{eff}}(K)
}{
\log \Lambda_{\mathcal U_\iota}
}
\]
to $\mathscr D(\Xi_\iota,\hat{\deg}_{\Xi_\iota})$.
Hence the right-hand side is contained in the degree spectrum.

Conversely, let $s\in \mathscr D(\Xi_\iota,\hat{\deg}_{\Xi_\iota})$.
By definition, there exists a sequence $\ell_n\downarrow 0$ such that
$P_{(\Xi_\iota,\hat{\deg}_{\Xi_\iota})}(\ell_n)>0$ and
\[
\frac{
\log P_{(\Xi_\iota,\hat{\deg}_{\Xi_\iota})}(\ell_n)
}{
-\log \ell_n
}
\to s.
\]
All sufficiently small positive exact degree levels belong to some class in $\mathscr K_{\deg}(\iota)$.
Because there are only finitely many classes, we may pass to a subsequence and assume that
\[
\ell_n=\alpha_K\Lambda_{\mathcal U_\iota}^{-m_n}
\]
for one fixed class $K$.
Lemma~\ref{lem:class-count} then gives
\[
\log P_{(\Xi_\iota,\hat{\deg}_{\Xi_\iota})}(\ell_n)
=
m_n\log \Lambda_{\mathbf M}^{\deg,\mathrm{eff}}(K)+O(\log m_n)
\]
and
\[
-\log \ell_n
=
m_n\log \Lambda_{\mathcal U_\iota}+O(1).
\]
Therefore
\[
s
=
\frac{
\log \Lambda_{\mathbf M}^{\deg,\mathrm{eff}}(K)
}{
\log \Lambda_{\mathcal U_\iota}
}.
\]
This proves the reverse inclusion.

Hence
\[
\mathscr D(\Xi_\iota,\hat{\deg}_{\Xi_\iota})
=
\bigg\{
\frac{
\log \Lambda_{\mathbf M}^{\deg,\mathrm{eff}}(K)
}{
\log \Lambda_{\mathcal U_\iota}
}
:
K\in\mathscr K_{\deg}(\iota)
\bigg\}.
\]
Since $\mathscr K_{\deg}(\iota)$ is finite, the degree spectrum is finite.
It is non-degenerate if and only if there exist two classes
$K_1,K_2\in\mathscr K_{\deg}(\iota)$ such that
$\Lambda_{\mathbf M}^{\deg,\mathrm{eff}}(K_1)\ne \Lambda_{\mathbf M}^{\deg,\mathrm{eff}}(K_2)$, which is exactly the BEDM condition.
\end{proof}

\begin{example}
We revisit the splendor diamond hierarchical lattice introduced in the opening section;
see Figure~\ref{fig:splendor}.
Its substitution matrices for mass and degree are
\[
\mathbf M=
\begin{pmatrix}
2 & 1 & 1\\
0 & 4 & 0\\
0 & 0 & 5
\end{pmatrix},
\qquad
\mathbf N=
\begin{pmatrix}
2 & 0 & 0 & 0 & 0 & 0\\
0 & 0 & 0 & 1 & 0 & 1\\
0 & 0 & 2 & 0 & 0 & 0\\
0 & 0 & 0 & 2 & 0 & 0\\
0 & 0 & 0 & 0 & 2 & 0\\
0 & 0 & 0 & 0 & 0 & 2
\end{pmatrix}.
\]
Here $\Lambda_{\mathcal U_\iota}=2$.
The dominant degree layer splits into two classes, represented by the interior birth types of degrees $2$ and $3$.
Their effective mass growth rates are $4$ and $5$, respectively.
Therefore this EIGS satisfies the BEDM condition.
Theorem~\ref{thm:multiscalefree} yields
\[
\mathscr D(\Xi_\iota,\hat{\deg}_{\Xi_\iota})
=
\bigg\{
\frac{\log 4}{\log 2}=2,\,
\frac{\log 5}{\log 2}\approx 2.3219
\bigg\}.
\]
The same two mass rates also govern the surviving fractal branches, so in this example
\[
\mathscr D(\Xi_\iota,\hat{\deg}_{\Xi_\iota})
=
\mathscr S \ximetric.
\]

A useful rule of observation is that, when several straight segments appear in the degree distribution plot,
their point of intersection can hint at the underlying regime.
If the segments meet at small degree values, such as in Figure~\ref{fig:scale-free-simulation},
then the extra branches are likely to vanish in the degree scaling limit, so the graph is merely scale-free.
By contrast, if they intersect at large degree values, as in Figure~\ref{fig:multi-scale},
then the tail retains more than one branch as $\ell\to 0$, producing several $\omega$-limits;
such a pattern is typical of a multiscale-free graph.

Figure~\ref{fig:multi-scale} shows how our simulations for { $\Xi_1^7$ and $\Xi_1^{11}$}
confirm that the system indeed exhibits a multiscale-free structure with a discrete degree spectrum.
In particular, the double-log regression for $\Xi^{11}$, with slopes being $1.95$ and $2.34$, yields two empirical branches close to the theoretical values
$2$ and $2.3219$.
In fact, once we restrict attention to the data points with $\ell\to 0$,
the theoretical prediction and the simulation agree almost perfectly.

\begin{figure}[htb]
  \centering

  \begin{minipage}[b]{0.49\textwidth}
    \centering
    \includegraphics[width=\linewidth]{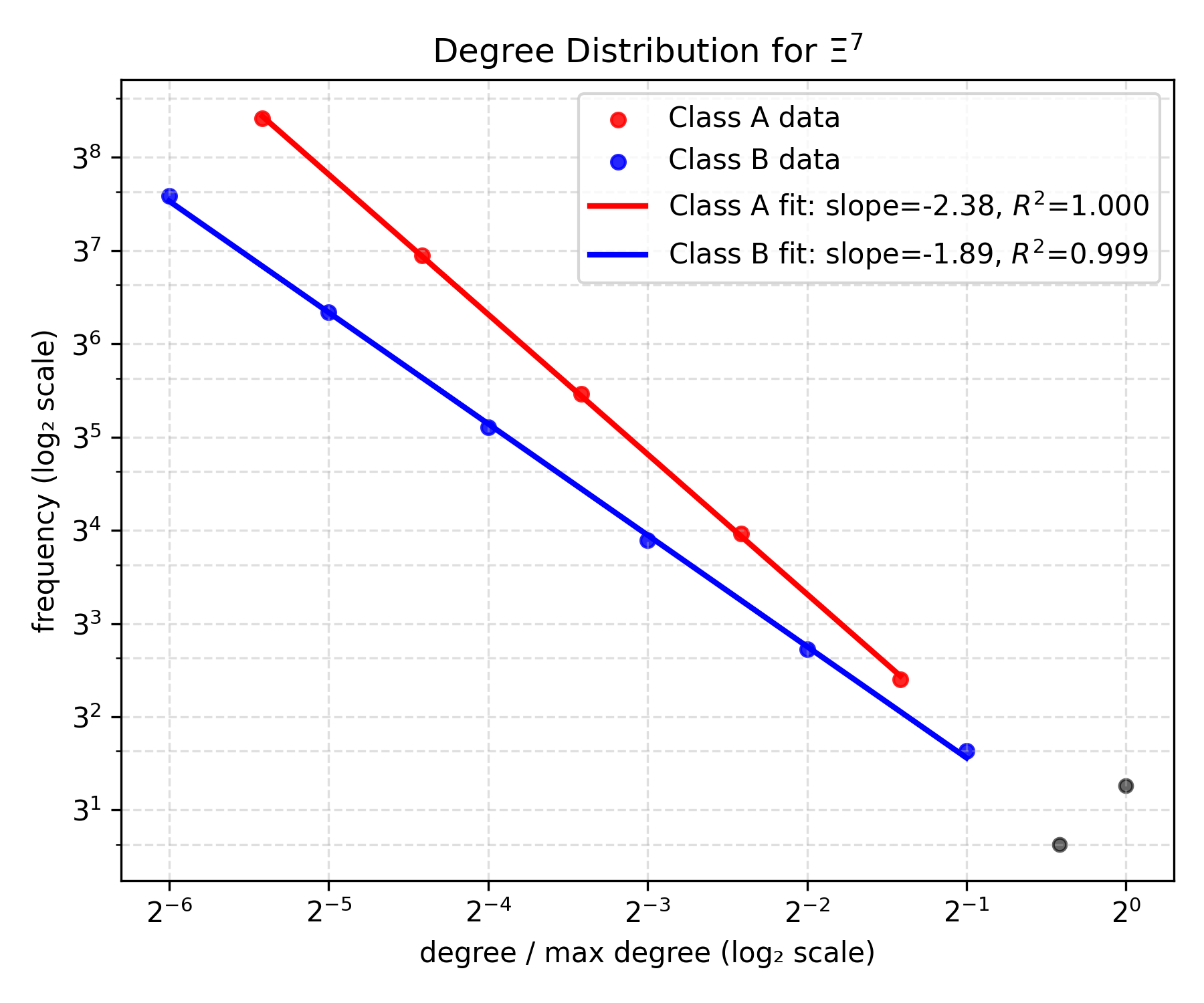}
  \end{minipage}
  \hfill
  \begin{minipage}[b]{0.49\textwidth}
    \centering
    \includegraphics[width=\linewidth]{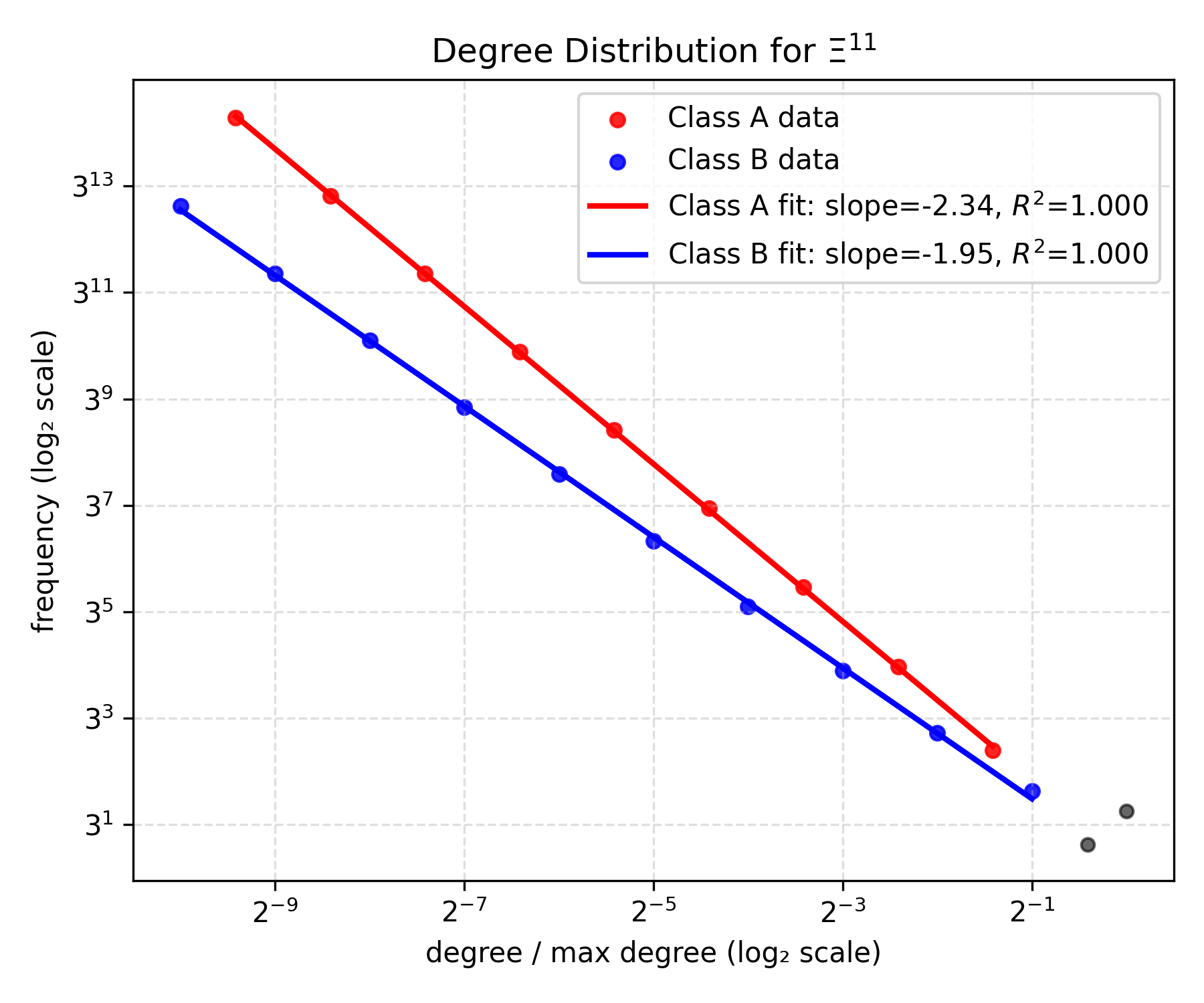}
  \end{minipage}

  \caption{Numerics of the multiscale-freeness of Splendor DHL $\Xi_1^{7}$ and $\Xi_1^{11}$.}
  \label{fig:multi-scale}
\end{figure}
\end{example}

\section*{Acknowledgments}
The authors would like to thank Koo Feng Goh, Xuanyu Yu and Xiaole Zou for their comments.
The work was supported by the Additional Funding Programme for Mathematical Sciences, 
delivered by EPSRC (EP/V521917/1) and the Heilbronn Institute for Mathematical Research to [N.Z.L.]. 
This work was also supported by the EPSRC Centre for Doctoral Training in Mathematics of Random Systems: 
Analysis, Modelling and Simulation (EP/S023925/1) to [N.Z.L.].



\section*{References}
\bibliography{References}

\end{document}